\documentclass[fleqn]{amsart}
\usepackage{hyperref}
\usepackage{amssymb}
\usepackage{enumitem}
\usepackage{xcolor}
\usepackage{graphicx}
\usepackage[Symbol]{upgreek}
\renewcommand{\alpha}{\upalpha}
\renewcommand{\beta}{\upbeta}
\renewcommand{\gamma}{\upgamma}
\renewcommand{\delta}{\updelta}
\renewcommand{\epsilon}{\upepsilon}
\renewcommand{\zeta}{\upzeta}
\renewcommand{\eta}{\upeta}
\renewcommand{\theta}{\uptheta}
\renewcommand{\iota}{\upiota}
\renewcommand{\kappa}{\upkappa}
\renewcommand{\lambda}{\uplambda}
\renewcommand{\mu}{\upmu}
\renewcommand{\nu}{\upnu}
\renewcommand{\xi}{\upxi}
\renewcommand{\pi}{\uppi}
\renewcommand{\rho}{\uprho}
\renewcommand{\sigma}{\upsigma}
\renewcommand{\tau}{\uptau}
\renewcommand{\upsilon}{\upupsilon}
\renewcommand{\phi}{\upphi}
\renewcommand{\chi}{\upchi}
\renewcommand{\psi}{\uppsi}
\renewcommand{\omega}{\upomega}
\usepackage{mathrsfs} 

\newcommand{\frL}{\mathfrak{L}}
\newcommand{\frM}{\mathfrak{M}}

\newcommand{\rF}{\mathrm{F}}

\newcommand{\rH}{\mathrm{H}}

\newcommand{\cA}{\mathcal{A}}

\newcommand{\cH}{\mathcal{H}}
\newcommand{\cI}{\mathcal{I}}

\newcommand{\cK}{\mathcal{K}}

\newcommand{\cP}{\mathcal{P}}

\newcommand{\cS}{\mathcal{S}}

\newcommand{\cU}{\mathcal{U}}
\newcommand{\cV}{\mathcal{V}}
\newcommand{\cW}{\mathcal{W}}

\newcommand{\scrA}{\mathscr{A}}

\newcommand{\scrI}{\mathscr{I}}

\newcommand{\scrP}{\mathscr{P}}

\makeatletter
\count@=`A \advance\count@\m@ne
\@whilenum\count@<`Z\do{%
  \advance\count@\@ne
  \begingroup\uccode`a=\count@
  \uppercase{\endgroup\DeclareMathSymbol{a}}{\mathalpha}{operators}{\count@}%
}
\makeatother

\newcommand{\C}{\mathbb{C}}
\newcommand{\N}{\mathbb{N}}
\newcommand{\Q}{\mathbb{Q}}
\newcommand{\R}{\mathbb{R}}
\newcommand{\Z}{\mathbb{Z}}
\newcommand{\ud}{\mathrm{d}}

\DeclareMathOperator{\card}{card}

\DeclareMathOperator{\GL}{GL}
\DeclareMathOperator{\SL}{SL}

\DeclareMathOperator{\SA}{SA}

\DeclareMathOperator{\tr}{tr}

\DeclareMathOperator{\diag}{diag}

\DeclareMathOperator{\Aff}{Aff}

\DeclareMathOperator{\Mat}{M}

\newcommand{\la}{\langle}
\newcommand{\ra}{\rangle}

\DeclareMathOperator{\Stab}{Stab}
\newcommand{\F}{\mathbb{F}}
\DeclareMathOperator{\Spec}{Spec}
\DeclareMathOperator{\First}{First}
\DeclareMathOperator{\Last}{Last}
\newcommand{\Vv}{\mathcal{V}}
\newcommand{\Ww}{\mathcal{W}}
\newcommand{\Nn}{\mathcal{N}}
\newcommand{\Uu}{\mathcal{U}}

\newcommand{\TV}{\mathrm{TV}}

\DeclareMathOperator{\Ind}{Ind}
\newcommand{\T}{\mathbb{T}}

\DeclareMathOperator*{\BigCup}{{\textstyle\bigcup}}
\renewcommand{\P}{\mathbb{P}}
\newcommand{\G}{\mathbf{G}}
\renewcommand{\H}{\mathbf{H}}
\newcommand{\rmod}{\ \mathrm{mod}\ }
\newcommand{\Ss}{\mathcal{S}}

\newcommand{\Leb}{\mathrm{Leb}}
\newtheorem{theorem}{Theorem}[section]
\newtheorem{lemma}[theorem]{Lemma}
\newtheorem{proposition}[theorem]{Proposition}
\newtheorem{corollary}[theorem]{Corollary}
\newtheorem{definition}[theorem]{Definition}

\newtheorem{problem}{Problem}
\newtheorem{claim}{Claim}

\theoremstyle{definition}
\newtheorem{remark}[theorem]{Remark}

\numberwithin{equation}{section}

\title{Uniform Kazhdan Constants and Paradoxes of the Affine Plane}
\author{Lam L. Pham}
\date{\today}
\email{lam.pham@yale.edu}
\address{Yale University}

\begin{document}

\begin{abstract}
Let $G=\SL(2,\Z)\ltimes\Z^2$ and $H=\SL(2,\Z)$. We prove that the action $G\curvearrowright\R^2$ is \emph{uniformly non-amenable} and that the quasi-regular representation of $G$ on $\ell^2(G/H)$ has a \emph{uniform spectral gap}. Both results are a consequence of a uniform quantitative form of ping-pong for affine transformations, which we establish here.
\end{abstract}

\maketitle

\section{Introduction}

\subsection{Kazhdan's Property $(T)$}

Let $G$ be a countable group, and let $S\subset G$ be a finite set.  Given a unitary representation $(\pi,\cH)$ of $G$, the \emph{Kazhdan constant} (or \emph{spectral gap}) of $\pi$ relative to $S$ is defined as
\[
\kappa_G(S,\pi)= \inf \left\{ \sup_{g\in S}\|\pi(g)\xi-\xi\|\,:\,
\xi\in\cH_\pi,\
\|\xi\|=1
\right\}.
\]
If $H\leq G$ is a subgroup, we denote by $\cH^H$ the subspace of $H$-invariant vectors. We say that $G$ has \emph{Kazhdan's Property $(T)$} if there exists a finite set $S$ generating $G$ (henceforth, the group generated by $S$ will be denoted by $\la S\ra$) such that $\kappa_G(S)= \inf_\pi \kappa_G(S,\pi)>0$, where the infimum is taken over all unitary representations $(\pi,\cH)$ of $G$ such that $\cH^G=\{0\}$.

An open problem first put forth by Lubotzky \cite{Lubotzky1994} is to determine for which groups $\inf_S \kappa_G(S)>0$, where the infimum is taken over all finite sets $S$ generating $G$. Such a group will be called \emph{uniform Kazhdan}. When focusing on a specific representation $\pi$, let us write $\kappa_G(\pi)=\inf_S \kappa_G(S,\pi)$. 

Gelander and \.{Z}uk \cite{GelanderZuk2002} showed that a finitely generated group admitting a dense embedding in a connected Lie group cannot be uniform Kazhdan; this includes irreducible lattices in products of at least two Lie groups. On the other hand,  Osin and Sonkin \cite{OsinSonkin2007} managed to construct finitely generated uniform Kazhdan groups. While $\SL(3,\Z)$ does have Property $(T)$, the problem of determining whether it is uniform Kazhdan remains open \cite{BekkaHarpeValette2008}.

According to the \emph{Tits alternative} \cite{Tits1972}, every finitely generated linear group is either virtually solvable, or contains a subgroup isomorphic to the non-abelian free group on two generators $\mathrm{F}_2$. Building on the work of Eskin, Mozes, and Oh \cite{EskinMozesOh2005}, Breuillard and Gelander \cite{BreuillardGelander2008} showed that the Tits alternative could be made effective and uniform in the following sense: there exists $N\in\N$ such that for any finite symmetric generating set $S$ (i.e. $S^{-1}=S$) containing $1$, $S^N$ contains two generators of $\mathrm{F}_2$. Recall that a discrete group $G$ is \emph{uniformly non-amenable} if $\kappa_G(\lambda_G)>0$ where $\lambda_G$ is the left regular representation of $G$ on $\ell^2(G)$; this was first investigated by Shalom \cite{Shalom2000} and Osin \cite{Osin2002a}, and a slightly different definition was given in \cite{ArzhantsevaBurilloLustigReevesShortVentura2005}. One of the key applications of the uniform Tits alternative is precisely to show that non-virtually solvable finitely generated linear groups are uniformly non-amenable \cite[Theorem 8.1]{BreuillardGelander2008}.

In this paper, we study the following generalization of uniformly non-amenable groups. Let us say that a measurable action $G\curvearrowright(X,\frM)$ is \emph{$(S,\epsilon)$-non-amenable} for some finite subset $S\subset G$ and $\epsilon>0$, if for every finitely additive probability measure $\mu$ on $\frM$,
\[
\sup_{g\in S} \|g_*\mu-\mu\|_{\TV}>\epsilon,
\]
where if $\mu$ and $\nu$ are two finitely additive probability measures on $(X,\frM)$, $\|\mu-\nu\|_{\TV}=\sup_{A\in\frM} |\mu(A)-\nu(A)|$, and $g_*\mu$ is the pushforward measure of $\mu$ by $g$. We define the action $G\curvearrowright(X,\frM)$ to be \emph{uniformly non-amenable} if there exists $\epsilon>0$ such that it is $(S,\epsilon)$-non-amenable for every finite  generating set $S$.

When $(X,\frM)$ is a countable discrete space, if the action $G\curvearrowright X$ is uniformly non-amenable, then $\kappa_G(\pi_X)>0$, where $\pi_X$ is the natural representation of $G$ on $\ell^2(X)$ acting by left translations, defined by
\[
[\pi_X(g)f](x)=f(g^{-1}x),\quad g\in G,\,x\in X,\, f\in \ell^2(X)
\]
(see Proposition \ref{uniformly-non-amenable-Kazhdan}). Note also that if $X=G/H$ for some subgroup $H\leq G$, then $\pi_{G/H}$ is precisely the quasi-regular representation $\lambda_{G/H}$ of $G$ on $\ell^2(G/H)$; in particular, $\pi_G=\lambda_G$ is the regular representation of $G$ on $\ell^2(G)$.

\subsection{Relative Kazhdan's Property $(T)$ and Expanders}

It is known that the groups $\SL(d,\Z)$ and $\SA(d,\Z)=\SL(d,\Z)\ltimes\Z^d$ have Property $(T)$ for $d\geq 3$, so one may also ask whether $\SA(d,\Z)$ is uniform Kazhdan. On the other hand, while neither $\SL(2,\Z)$ nor $\SA(2,\Z)$ has Property $(T)$, the pair $(\SA(2,\Z),\Z^2)$ does have the \emph{relative} Property $(T)$ \cite{Kazhdan1967,Burger1991,Shalom1999}: there exists a finite generating set $S$ such that $\inf_\pi\kappa_{\SA(2,\Z)}(S,\pi)>0$, where the infimum is taken over all unitary representations without $\Z^2$-invariant vectors. If $G$ is a finitely generated group and $H\leq G$ is a subgroup, we will call the pair $(G,H)$ \emph{uniform Kazhdan} if $\inf_{S,\pi} \kappa_G(S,\pi)>0$, where the infimum is taken over all finite generating sets of $G$ and all unitary representations $(\pi,\cH)$ of $G$ such that $\cH^H=\{0\}$.

\begin{problem}[Uniform Relative $(T)$]
\label{relative-uniform-Kazhdan}
Is $(\SA(2,\Z),\Z^2)$ a uniform Kazhdan pair?
\end{problem}

Property $(T)$ for the pair $(\SA(2,\Z),\Z^2)$ enabled Margulis \cite{Margulis1973} to give the first construction of \emph{expander graphs}, and is crucial in the computation of explicit Kazhdan constants for $\SL(3,\Z)$ by Burger \cite{Burger1991} and Shalom \cite{Shalom1999}. Given a sequence $(H_n)_{n\in\N}$ of finite-index subgroups of a finitely generated group $G=\la S\ra$ with $[G:H_n]\to \infty$ as $n\to\infty$, recall that $(G,S,(H_n)_n)$ is an \emph{expander family} if there exists $\epsilon(S)>0$ such that $\inf_{n\in\N}\kappa_G(S,\pi_{n}^0)\geq\epsilon(S)$, where $\pi_{n}^0$ is the restriction of the quasi-regular representation $\lambda_{G/H_n}$ to the subspace $\ell_0^2(G/H_n)$ orthogonal to the constants. In particular, if $G$ has Property $(T)$, then $(G,S,(H_n)_n)$ is an expander family; using Property $(T)$ for the pair $(\SA(2,\Z),\Z^2)$ proved by Kazhdan \cite{Kazhdan1967}, Margulis \cite{Margulis1973} showed that $(G,S,(H_n)_n)$ is an expander family where $G=\SA(2,\Z)$ and $H_n=\SL(2,\Z)\ltimes(n\Z)^2$ for each $n\in\N$.

The \emph{independence problem} of Lubotzky and Weiss \cite{LubotzkyWeiss1993} asks whether expansion is a group property in the following sense: let $(G_n)_n$ be a sequence of finite groups, and $(S_n)_n$, $(S'_n)_n$ a sequence of finite generating subsets of fixed cardinality. Let $\lambda_n=\lambda_{G_n}$ denote the regular representation of $G_n$. If $\inf_n\kappa_{G_n}(S_n,\lambda_n^0)>0$, does it necessarily follow that $\inf_n\kappa_{G_n}(S_n',\lambda_n^0)>0$? While several counterexamples were constructed by Alon, Lubotzky, and Wigderson \cite{AlonLubotzkyWigderson2001}, the independence problem for the sequence of groups $(\SL(2,\F_p))_{p}$, where $p$ runs over all primes, remains open. Since the groups $\SL(2,\F_p)$ arise as finite quotients of $\SL(2,\Z)$, one may formulate the following analogue of the independence problem, which we call \emph{uniform Property $(\tau$)}.

\begin{problem}\label{uniform-tau}
Let $G$ be a finitely generated group and $(H_n)_n$ a sequence of finite index normal subgroups. Is it true that $\inf_n \kappa_G(\lambda_{G/H_n}^0)>0$?
\end{problem}

If the answer to Problem \ref{uniform-tau} for a given family $(G,(H_n)_n)$ is positive, then we call this family a \emph{uniform expander family}. Clearly, any infinite, residually finite uniform Kazhdan group gives rise to such a family, but the existence of such a group remains elusive \cite{OsinSonkin2007}. Nonetheless, Breuillard and Gamburd \cite{BreuillardGamburd2010} showed that for $G=\SL(2,\Z)$, there is a density one set of primes $\scrP_1\subset\scrP$ such that $(G,(H_p)_{p\in\scrP_1})$ forms a uniform expander family, where $H_p=\ker(\SL(2,\Z)\to\SL(2,\F_p))$. On the other hand, Lindenstrauss and Varj\'u \cite{LindenstraussVarju2016} proved that for any $\scrA\subseteq\scrP$, if $(\SL(2,\Z),(H_p)_{p\in\scrA})$ is a uniform expander family, then $(\SA(2,\Z),(H'_p)_{p\in\scrA})$ is a uniform expander family, where $H'_p=\ker(\SA(2,\Z)\to \SA(2,\F_p))$. A positive answer to Problem \ref{relative-uniform-Kazhdan} would yield this same statement without the restriction that $p$ be prime, namely, that for any subset $\scrA\subseteq\N$, if $(\SL(2,\Z),(H_n)_{n\in\scrA})$ is a uniform expander family, then $(\SA(2,\Z),(H_n')_{n\in\scrA})$ is a uniform expander family, where $H_n=\ker(\SL(2,\Z)\to \SL(2,\Z/n\Z))$ and $H'_n=\ker(\SA(2,\Z)\to\SA(2,\Z/n\Z))$ for each $n\in\N$. It would also imply that the expander family constructed by Margulis \cite{Margulis1973} can be made in fact uniform, i.e., that $(G,(H_n)_{n\in\N})$ is a uniform expander family, where $G=\SA(2,\Z)$ and $H_n=\SL(2,\Z)\ltimes(n\Z)^2$.

\subsection{Main Results}

Let $G\curvearrowright(X,\frM)$ be a measurable action. Let $E\in\frM$ and $S\subseteq G$. Let us say that $E$ is \emph{$(S,n+m)$-paradoxical} if there exist a finite index set $\cI$, a partition $\cI=\cI_1\sqcup \cI_2$ with $|\cI_1|=n$, $|\cI_2|=m$, elements $g_i\in S$ and pairwise disjoint measurable subsets $A_i\subset E$ for every $i\in\cI$, such that $E=\BigCup_{i\in\cI_1}g_i A_i=\BigCup_{i\in \cI_2} g_i A_i$. Given an integer $r\geq 4$, we say that $E$ is \emph{$G$-paradoxical with $r$-pieces} if it is $(G,r)$-paradoxical.

In his thesis, Dekker \cite{Dekker1956,Dekker1957,Dekker1958} defined an action to be \emph{locally commutative} if the stabilizer of every point is commutative; equivalently, any two group elements with a common fixed point must commute. The main result of this paper is the following. 

\begin{theorem}[Main Theorem]\label{main-theorem}
There exists $N\in\N$ such that for any finite symmetric set $S\subset \SA(2,\Z)$ containing $1$ and generating a non-virtually solvable subgroup $\Gamma$ which does not have a global fixed point in $\Q^2$, $S^N$ contains two elements freely generating a non-abelian free group $\mathrm{F}_2$ whose natural action on $\R^2$ is locally commutative.
\end{theorem}

We now state some consequences of our main theorem. The first one regards the existence of paradoxical decompositions. Dekker proved that a set $X$ is $G$-paradoxical using $4$ pieces if and only if $G$ contains an isomorphic copy of $\mathrm{F}_2$ whose action on $X$ is locally commutative. The following corollary of Theorem \ref{main-theorem} shows that paradoxical decompositions can be quickly found, regardless of the choice of generators.

\begin{corollary}[Paradoxical decompositions]\label{corollary-paradoxical}
There exists $N\in\N$ such that for any finite symmetric set $S\subset \SA(2,\Z)$ containing $1$ and generating a non-virtually solvable subgroup $\Gamma$ which does not have a global fixed point in $\Q^2$, there exist $a,\,b\in S^N$ such that the plane $\R^2$ is $(\{1,a,b\},4)$-paradoxical.
\end{corollary}

Finding a free subgroup $\mathrm{F}_2$ in $\SA(2,\R)$ whose action on $\R^2$ is locally commutative was an open problem in Wagon's book \cite[Problem 19(c) p.233]{Wagon1993} regarding paradoxical decompositions, and was solved by Sat\^o \cite{Sato2003} by explicitly constructing such generators of $\mathrm{F}_2$. Recently, Breuillard, Green, Guralnick, and Tao \cite[Appendix C]{BreuillardGreenGuralnickTao2015} gave a geometric proof of this result using a ping-pong argument valid for $\SA(2,k)$ for any local field $k$. A refinement of their argument is a key new ingredient in proving Theorem \ref{main-theorem}.

Free subgroups of $\Aff(\R^2)$ never act freely on $\R^2$, because any affine transformation whose linear part does not have $1$ as an eigenvalue must fix a point. Thus, local commutativity is the best one can hope for. Note however that, in dimension $d\geq 3$, in connection with the \emph{Auslander conjecture} \cite{Auslander1964} and a question of Milnor \cite{Milnor1977}, Margulis \cite{Margulis1983,Margulis1984a} constructed free subgroups of $\SA(d,\R)$ acting properly discontinuously on $\R^d$. 

According to a famous theorem of Tarski \cite{Tarski1929,Tarski1938}, a set $E\subseteq X$ is not $G$-paradoxical if and only if there exists a finitely-additive measure $\mu$ on $\cP(X)$ with $\mu(E)=1$. With this in mind, Corollary \ref{corollary-paradoxical} implies the following.

\begin{corollary}[Uniform non-amenability]\label{uniform-non-amenable-affine-action}
There exists $N\in\N$ such that for any finite symmetric set $S\subset\SA(2,\Z)$ containing $1$ and generating a non-virtually solvable subgroup $\Gamma$ which does not have a global fixed point, there exist $a,\,b\in S^N$ such that the action of $\Gamma$ on $\R^2$ endowed with its Borel $\sigma$-algebra, is $(\{a,b\},1/4)$-non-amenable. In particular, $\Gamma\curvearrowright \R^2$ is uniformly non-amenable.
\end{corollary}

Note that in particular, the action of $\Gamma$ on $\Z^2$ has this property. Uniform Kazhdan constants for $\SA(2,\Z)$ were our initial motivation, and we now give several consequences of our main result.

\begin{corollary}[Uniform Kazhdan constants I]\label{uniform-Kazhdan}
There exists $\epsilon>0$ such that for any non-virtually solvable subgroup $G$ of $\SA(2,\Z)$ which does not have a global fixed point, and any subgroup $H\leq G$ which is not Zariski dense in $G$, we have $\kappa_G(\lambda_{G/H})>\epsilon$.
\end{corollary}

In Section \ref{sec:uniform-Kazhdan}, we will discuss Problem \ref{relative-uniform-Kazhdan}. As we will also recall there, while it follows from \cite{Burger1991,BreuillardGelander2008} that every representation of $\SA(2,\Z)$ coming from a representation of $\SA(2,\R)$ without $\R^2$-invariant vectors admits a positive uniform Kazhdan constant (see Proposition \ref{Lie-group-uniform-spectral-gap}), it turns out that Theorem \ref{main-theorem} provides several new classes of irreducible unitary representations for which such a uniform bound holds, as highlighted by the following Corollary.

\begin{corollary}[Uniform Kazhdan constants II]\label{uniform-Kazhdan-bound-II}
Let $\scrI$ be the class of irreducible representations of $G=\SA(2,\Z)$ that are induced from an irreducible unitary representation of $\SL(2,\Z)$. Then, $\inf_{\pi\in\scrI} \kappa_G(\pi)>0$.
\end{corollary}

\subsection{Outline of the Article}

As previously mentioned, the original motivation for Theorem \ref{main-theorem} comes from an attempt to address Problem \ref{relative-uniform-Kazhdan}, but Theorem \ref{main-theorem} is itself of independent interest. To prove Theorem \ref{main-theorem}, we rely on the techniques developed in \cite{BreuillardGelander2003,BreuillardGelander2008} and refine them to analyze the affine action $\SA(2,\Z)\curvearrowright\R^2$, via an effective and uniform elaboration of the ping-pong argument of \cite{BreuillardGreenGuralnickTao2015} to find generators of a free subgroup whose action on the plane is locally commutative.

In Section \ref{sec-notations-preliminaries}, we set up the notations and recall the relevant background. In particular, we will explain how Corollaries \ref{corollary-paradoxical} and \ref{uniform-non-amenable-affine-action} easily follow from Theorem \ref{main-theorem}. In Section \ref{sec:ping-pong}, we derive our abstract and quantitative Ping-Pong Lemmas, which we use in Section \ref{sec:proof-main-theorem-uniform-non-amenability} in combination with arithmeticity of $\SA(2,\Z)$ to prove Theorem \ref{main-theorem}. Finally, in Section \ref{sec:uniform-Kazhdan}, we prove Corollaries \ref{uniform-Kazhdan} and \ref{uniform-Kazhdan-bound-II} and discuss Problem \ref{relative-uniform-Kazhdan}.

\subsection*{Acknowledgments}

The questions investigated in this paper were raised during the author's PhD thesis at Yale University under the supervision of Emmanuel Breuillard and Gregory Margulis. I would like to thank them both for their guidance and valuable conversations; in particular, I would like to thank Emmanuel Breuillard from whom I first learned about Problem \ref{relative-uniform-Kazhdan}. I would also like to thank Alex Lubotzky for his constant support and encouragements. I am also very grateful to Emmanuel Breuillard for a careful reading of the manuscript and many comments which greatly improved the exposition.

\section{Notations and Preliminaries}\label{sec-notations-preliminaries}

\subsection{Group of Affine transformations}

Let $\Aff(\R^2)$ be the group of affine transformations of $\R^2$. We parametrize an element $g\in\Aff(\R^2)$ by its linear part $\theta(g)\in\GL(\R^2)$, and its translation part $\tau(g)\in\R^2$. By definition, $g\in\Aff(\R^2)$ acts on $\R^2$ by
\[
gx=\theta(g)x+\tau(g),\quad\forall x\in\R^2,\quad \theta(g)\in\GL(\R^2),\quad \tau(g)\in\R^2.
\]
The group $\Aff(\R^2)$ can then be described as the semidirect product $\Aff(\R^2)=\GL(\R^2)\ltimes\R^2$ and the natural quotient map $\theta:\Aff(\R^2)\to\GL(\R^2)$ is a group homomorphism. We endow $\R^3$ with its natural Euclidean norm, and let $\|\cdot\|$ denote the induced operator norm on $\GL(\R^3)$. The group $\Aff(\R^2)$ can be embedded as a subgroup of $\GL(3,\R)$ via the embedding
\[
\iota:\Aff(\R^2)\hookrightarrow \GL(3,\R),\quad
g\mapsto
\begin{pmatrix}
\theta(g)&\tau(g)\\0&1
\end{pmatrix},
\]
and this embedding allows us to define the norm on $\Aff(\R^2)\cong \GL(2,\R)\ltimes\R^2$ as being the operator norm inherited from the one on $\GL(3,\R)$. We also write $\SA(2,\R)=\SL(2,\R)\ltimes\R^2$.

\subsection{Joint Spectral Radius and Related Quantities}

Let $\|\cdot\|$ denote the operator norm on $\Mat(2,\R)$ induced by the standard Euclidean norm on $\R^2$. For any $g\in\Mat(2,\R)$, we denote by $\Spec(g)\subset\C$ the set of its eigenvalues. For any bounded subset $S\subset\Mat(2,\R)$, we define the \emph{norm} of $S$ by $\|S\|=\sup\{ \|g\|\,:\, g\in S\}$, along with the following quantities:
\begin{align*}
E(S)&=\inf\big\{\|gSg^{-1}\|\,:\, g\in \GL(2,\R)\big\},&\text{the \emph{minimal norm} of $S$;}\\
\Lambda(S)&=\max\big\{|\lambda|\,:\, \lambda\in \Spec(q),\,q\in S\big\},&\text{the \emph{maximal eigenvalue} of $S$.}
\end{align*}

\subsection{The Spectral Radius Lemma}

All these quantities relate thanks to the following \emph{Spectral Radius Lemma}. For convenience of the reader, we give below the proof of \cite[Lemma 2.1]{Breuillard2011a}. 
A closely related result for $\Mat(2,\C)$ was proved by Bochi \cite{Bochi2003} with a completely different proof, and a stronger version for an arbitrary local field was proved by Breuillard \cite{Breuillard2011}. 

\begin{lemma}[Spectral Radius Lemma]\label{spectral-radius-lemma}
There exists $c\in(0,1)$ such that for any bounded subset $S\subset\Mat(2,\R)$ containing $1$, $E(S^2)\geq\Lambda(S^2)\geq c^2\, E(S)^2$.
\end{lemma}

\begin{proof}
Assume by contradiction that the claim fails: there exists a sequence $(S_n)_n$ of bounded subsets in $\Mat(2,\R)$ such that $\Lambda(S_n^2)/E(S_n)^2\to 0$ as $n\to\infty$. Letting $Q_n=S_n/E(S_n)$, we obtain a sequence $(Q_n)_n$ of subsets of $\Mat(2,\R)$ with $\Lambda(Q_n^2)\to 0$ as $n\to\infty$, and $E(Q_n)=1$ for all $n\in\N$. By compactness, we may pass to a subsequence and obtain a limit set $Q\subset\Mat(2,\R)$ such that $\Lambda(Q^2)=\Lambda(Q)=0$ and $E(Q)=1$. But $\Lambda(Q^2)=0$ implies that all the matrices in $Q^2$ (and thus in $Q\subset Q^2$) are nilpotent. In that case, $Q$ can be simultaneously triangularized by conjugation with an element of $\GL(2,\R)$: indeed, we may assume that $a,\,b\in\Mat(2,\R)\setminus\{0\}$. Then, $a^2=b^2=(ab)^2=0$ by nilpotence, which implies that all the kernels and images of $a$ and $b$ are equal to the same line in $\R^2$. Picking this direction and a complementary one yields the desired basis to conjugate $Q$. Once in upper triangular form, we may further conjugate $Q$ by a diagonal matrix $\diag(t,t^{-1})$ and let $t\to 0$, leading to $E(Q)=0$, a contradiction.
\end{proof}

\subsection{Eskin-Mozes-Oh's ``Escape from Subvarieties''}

An important tool in proving the uniform Tits alternative is a result of Eskin, Mozes, and Oh \cite[Proposition 3.2]{EskinMozesOh2005}, enabling one to ``escape proper subvarieties in a bounded number of steps''.  We will repeatedly make use of this result (see also \cite[Lemma 4.2]{Breuillard2011}, and \cite{BreuillardGreenTao2011} for an alternative proof).

Given an algebraic variety $X$, we denote by $m(X)$ the sum of the degree and the dimension of its irreducible components.

\begin{lemma}[Escape from Subvarieties]\label{Eskin-Mozes-Oh}\label{escape}
Let $d\in\N^*$. For every $r\in\N$, there exists $N(d,r)\in\N$ such that if $X\subseteq \GL(d,\R)$ is a subvariety such that $m(X)\leq r$, then for any subset $S\subset\GL(d,\R)$ containing $1$ and such that $\la S\ra\not\subset X$, we have $S^N\not\subset X$.
\end{lemma}

\subsection{Arithmetic Spectral Radius Lemma}

We will need the following arithmetic variant of Lemma \ref{spectral-radius-lemma}; see \cite[Proposition 5.7]{BreuillardGelander2008} for the proof of a more general result.

\begin{proposition}[Arithmetic Spectral Radius Lemma]\label{arithmetic-spectral-radius-lemma}
There exist $r\in\N^*$ and $c>0$ such that for any finite subset $S\subset\SL(2,\Z)$ containing $1$ and generating a non virtually solvable subgroup, there exists $\gamma\in\SL(2,\Z)$ such that
\[
\Lambda(S^r)\geq 2\,\|\gamma S\gamma^{-1}\|.
\]
\end{proposition}

\begin{proof}
Let $G=\SL(2,\R)$, $\Gamma=\SL(2,\Z)$, and let $\pi:G\to G/\Gamma$ be the natural map. We define $\|\pi(g)\|=\inf\{ \|g\gamma\|\,:\, \gamma\in\Gamma\}$. We first show that the proposition reduces to the following claim:

\begin{claim}\label{claim1}
There exist absolute constants $C_1,\,C_2>0$ such that for any finite subset $S\subset\SL(2,\Z)$ generating a non-virtually solvable subgroup,
\[
\|gSg^{-1}\|\geq C_1\,\|\pi(g)\|^{C_2},\quad\forall g\in G.
\]
\end{claim}
Indeed, by Lemma \ref{spectral-radius-lemma}, there exists $g\in G$ such that $\|gSg^{-1}\|\leq c^{-1}\,\Lambda(S^2)^{1/2}$. By Claim \ref{claim1}, there exists $\gamma\in\Gamma$ such that 
\begin{align*}
\|\gamma S \gamma^{-1}\|
&\leq \|\gamma g^{-1}\| \, \|gSg^{-1}\|\, \|g\gamma^{-1}\|
\leq C_1^{-2/C_2}\,c^{-1-\frac{2}{C_2}}\,
\Lambda(S^2)^{\frac{1}{2}+\frac{1}{C_2}}.
\end{align*}
The Zariski closure of the set of elements in $\la S\ra$ which are torsion (of order $\leq 6$) or not semisimple is a proper algebraic subvariety of $\SL(2,\R)$. Since $\la S\ra$ is not virtually solvable, $\la S\ra$ is not contained in that subvariety. By Lemma \ref{escape}, there exists $N_1$ independent of $S$ such that $S^{N_1}$ contains a torsion-free semisimple element $a$, so $\Lambda(S^{N_1})>2$. Choosing a large enough power (independent of $S$), we can get rid of all the above constants: there exists $N_2\in\N$ independent of $S$ such that $\Lambda(S^{N_2})\geq 2\|\gamma S\gamma^{-1}\|$, which proves the Proposition. Let us now show that Claim \ref{claim1}, reduces to Claim \ref{claim2} below, whose proof may be found in \cite[Lemma 5.7]{BreuillardGelander2008}.

\begin{claim}\label{claim2}
There exist $k,\,\ell>0$ such that for any $g\in G$, there exists a unipotent $u\in\Gamma\setminus\{1\}$ such that $\|gug^{-1}-1\|\leq \ell\, \|\pi(g)\|^{-k}$.
\end{claim}

Indeed, let $\epsilon>0$ be a constant to be chosen shortly. Assume by contradiction that Claim \ref{claim1} fails: assume that $\|gSg^{-1}\|<(\epsilon/\ell)^{1/6}\,\|\pi(g)\|^{k/6}$, where $k$ and $\ell$ are given by Claim \ref{claim2}. Let $u\in\Gamma\setminus\{1\}$ be a unipotent such that $\|gug^{-1}-1\|\leq \ell \|\pi(g)\|^{-k}$. Then, for any $w\in gS^3g^{-1}$, we have $\|(wg)u(wg)^{-1}-1\|<\epsilon$. 
For $1\leq i\leq 3$, let $S_i=gS^i u S^{-i} g^{-1}$. We can choose $\epsilon>0$ small enough that the \emph{Zassenhaus lemma} \cite[Theorem 8.16]{Raghunathan1972} holds: this implies that for $1\leq i\leq 3$, the groups $\la S_i\ra$ are nilpotent. Let $U_i$ denote the Zariski closure of $\la S_i\ra$, which is Zariski connected and nilpotent. Since the $U_i$'s are generated by unipotent elements, and the set of unipotent elements is Zariski closed in a Zariski connected solvable group \cite[Theorem 19.3]{Humphreys1975}, it follows that each $U_i$ is unipotent. Since $\dim(\SL(2,\R))=3$, the sequence of subgroups $U_1,\,U_2,\,U_3$ must stabilize, which shows that one of $U_1,\,U_2$ is normalized by $gSg^{-1}$. By a theorem of Borel-Tits \cite[Proposition 30.3]{Humphreys1975}, $gSg^{-1}$ is contained in a proper parabolic subgroup of $G$, which is solvable, a contradiction.\qedhere
\end{proof}

\subsection{Dynamics of Projective and Affine Transformations}

Let $\psi:\R^2\setminus\{0\}\to\P^1(\R)$ be the natural map, and for any $u\in\R^2\setminus\{0\}$, write $[u]=\psi(u)$. We endow the projective space $\P^1(\R)$ with the \emph{Fubini-Study distance} defined by
\[
d([u],[v])=\frac{\|u\wedge v\|}{\|u\|\,\|v\|},\quad\forall u,\,v\in \R^2\setminus\{0\}.
\]
This distance behaves reasonably well with respect to change of basis. Indeed, 
\begin{equation}\label{basic-estimates-distances-matrix}
\frac{|\det(g)|}{\|g\|^2}\leq
\frac{d([gu],[gv])}{d([u],[v])}
\leq
\frac{\|g\|^2}{|\det(g)|},\quad \forall u,\,v\in\R^2\setminus\{0\},\ \forall g\in \GL(2,\R).
\end{equation}

\subsection{Proofs of Corollaries \ref{corollary-paradoxical} and \ref{uniform-non-amenable-affine-action}}\label{sec:equivalence-of-properties}

To conclude this preliminary section, we explain how Corollaries \ref{corollary-paradoxical} and \ref{uniform-non-amenable-affine-action} easily follow from Theorem \ref{main-theorem}.

\subsubsection{}

Corollary \ref{corollary-paradoxical} follows from Theorem \ref{main-theorem} and the following more precise statement of Dekker's Theorem, whose proof may be found, e.g., in \cite[Theorem 5.5]{TomkowiczWagon2016}.

\begin{theorem}[\cite{Dekker1956,Dekker1957,Dekker1958}]\label{Dekker-theorem}
If the action of $\rF_2=\la a,b\ra$ on $X$ is locally commutative, then $X$ is $(\{1,a,b\},4)$-paradoxical.
\end{theorem}

\subsubsection{} 

The following is a quantitative version of the easy direction of Tarski's Theorem \cite{Tarski1929,Tarski1938}: it shows that paradoxical decompositions give rise to non-amenable actions, in a quantitative way.

\begin{lemma}[Paradoxical $\Rightarrow$ non-amenable]\label{paradoxical-implies-non-amenable}
Let $G\curvearrowright(X,\frM)$ be a measurable action and let $S\subset G$. If $X$ is $(S,n+m)$-paradoxical, then the action $G\curvearrowright (X,\frM)$ is $(S,(m+n)^{-1})$-non-amenable.
\end{lemma}

\begin{proof}
Let $S=\{g_i\,:\, i\in\cI\}$. Write $\cI=\cI_1\sqcup \cI_2$ and let $A_i\subset X$ ($i\in\cI$) be pairwise disjoint measurable subsets such that $X=\bigcup_{i\in\cI_1} g_iA_i=\bigcup_{i\in\cI_2} g_i A_i$. If $\mu$ is a finitely-additive probability measure on $\frM$, we have $\sum_{i\in \cI} \mu(A_i)\leq 1$, so
\begin{align*}
2&=\mu\Big(\bigcup_{ i\in \cI_1} g_iA_i\Big)+
\mu\Big(\bigcup_{ i\in \cI_2} g_iA_i\Big)
\leq 1+\sum_{i\in \cI} |\mu(g_i A_i)-\mu(A_i)|.\qedhere
\end{align*}
\end{proof}

The first part of Corollary \ref{uniform-non-amenable-affine-action} then follows from Corollary \ref{corollary-paradoxical}. For the second part, first notice that $\|g_*\mu-\mu\|_{\TV}=\|(g^{-1})_*\mu-\mu\|_{\TV}$ so we may assume that $S$ is symmetric, then apply the triangle inequality: for any $N\in\N$ and any finite set $S\subset G$,
\begin{equation}\label{eq:identity}
\sup_{g \in S^N}\|g_*\mu-\mu\|_{\TV}\leq N \,\cdot\, \sup_{g\in S} \|g_*\mu-\mu\|_{\TV}.
\end{equation}

\section{Ping Pong Lemmas}\label{sec:ping-pong}

In this section, we state two general ping-pong lemmas for group actions. The first is very standard, and the second is used to find generators of a free group whose action is locally commutative. We then establish quantitative versions of these ping-pong lemmas.

\subsection{Abstract Ping-Pong Lemmas}

A typical method to prove that two elements generate a free group consists in showing that they play ping-pong on some appropriate space; this was already used by Tits \cite{Tits1972} in the proof of his alternative. The following form is classical.

\begin{lemma}[Abstract Ping-Pong Lemma]\label{ping-pong-lemma-generic}
Let $G$ be a group acting on a set $X$. If there exist $a,\,b\in G$ and four disjoint non-empty subsets $A^+,\,A^-,\,B^+,\,B^-\subseteq X$, such that $
a\,(X\setminus A^{-})\subseteq A^{+}$ and $b\,(X\setminus B^{-})\subseteq B^{+}$, then $a$ and $b$ freely generate a non-abelian free subgroup $\mathrm{F}_2$.
\end{lemma}

When the conditions of the previous lemma hold, we say that $a$ and $b$ form a \emph{ping-pong pair}. For locally commutative actions, we have the following abstract ping-pong Lemma inspired from \cite[Appendix C]{BreuillardGreenGuralnickTao2015}. 

\begin{lemma}[Abstract Affine Ping-Pong Lemma]\label{general-affine-ping-pong-lemma}
Let $G$ act on a set $X$. Let $S=\{a,\,a^{-1},\,b,\,b^{-1}\}\subset G$. Assume that for each $x\in S$, we are given sets $\Uu_x^+$, $\Uu_x^-$, with the following properties:
\begin{align}
&y\,(X\setminus \Uu_y^-)\subseteq \Uu_y^+,\quad\forall y\in S,\label{eq:general-affine-ping-pong-1}\\
&\Uu_y^+\cap \Uu_z^-=\emptyset,\quad\forall z\neq y^{-1},\label{eq:general-affine-ping-pong-2}\\
&\Uu_x^-\cap \Uu_y^-\cap \Uu_z^-=\emptyset,\quad\forall x,\,y,\,z\in S,\quad\text{all distinct.}\label{eq:at-most-two}
\end{align}
Further, assume that there exists a function $f:X\to\R$ such that
\begin{align}
&f(sx)>f(x),\quad\forall x\in X\setminus \Uu_s^-,\quad\forall s\in S.\label{eq:dilation-length-one}
\end{align}
Then, $a$ and $b$ freely generate a non-abelian free group $\rF_2$ whose action on $X$ is locally commutative.
\end{lemma}

\begin{proof}
Let $w$ be a reduced word in $S$. Let $\First(w)$ and $\Last(w)$ denote the first and last letter of $w$, respectively. By induction on the length $\ell(w)$ of $w$, we see that
\begin{equation}\label{eq:ping-pong-proof-induction}
w\,(X\setminus \cU_{\Last(w)}^-)\subseteq
\cU_{\First(w)}^+,\quad \text{and}\quad
f(wx)>f(x),\quad\forall x\in X\setminus \cU_{\Last(w)}^-.
\end{equation}
Indeed, if $\ell(w)=1$, this is just \eqref{eq:general-affine-ping-pong-1} and \eqref{eq:dilation-length-one}, so assume $\ell(w)=n$ and that \eqref{eq:ping-pong-proof-induction} holds for all words $w'$ of length $\ell(w')\leq n-1$. Write $w=\First(w)\,w'$. The induction follows from the following implications:
\[
\begin{array}{lllll}
x\in X\setminus \cU_{\Last(w)}^-
&\ \overset{\eqref{eq:ping-pong-proof-induction}}{\Rightarrow}\
&w'x\in\cU_{\First(w')}^+
&\ \text{and}\
&f(w'x)>f(x)
\\
&\ \overset{\eqref{eq:general-affine-ping-pong-2}}{\Rightarrow}\
&w'x\in X\setminus \cU_{\First(w)}^-
&\ \text{and}\ 
&f(w'x)>f(x)
\\
&\ \overset{\eqref{eq:general-affine-ping-pong-1}}{\Rightarrow}\
&wx\in \cU_{\First(w)}^+
&\ \text{and}\
&f(wx)>f(x).
\end{array}
\]
For any word $w$ in $S$ and $x\in X\setminus \cU_{\Last(w)}^-$, we have $f(wx)>f(x)$, so the action of $w$ on $X$ is non-trivial. Thus, $a$ and $b$ generate a free group. \eqref{eq:ping-pong-proof-induction} also implies that any fixed point of a word $w$ must lie in $\cU_{\Last(w)}^-$, so by \eqref{eq:at-most-two}, any three reduced words in $\mathrm{F}_2$ with distinct last letter never share a common fixed point in $X$.

If $H$ is a set of words, put $\Last(H)=\{\Last(h)\,:\, h\in H\}$. A subgroup $H\leq \mathrm{F}_2$ is non-abelian if and only if there exists a word $w\in\mathrm{F}_2$ such that $|\Last(wHw^{-1})|\geq 3$ \cite[Lemma C.7]{BreuillardGreenGuralnickTao2015}. Let $x\in X$ and $H_x=\Stab_{F_2}(x)$. For any $w\in\mathrm{F}_2$, $H_{wx}=wH_xw^{-1}$, so if $H_x$ were non-abelian, there would exist $w\in F_2$ such that $|\Last(H_{w x})|=|\Last(wH_xw^{-1})|\geq 3$, which would imply that $H_{w x}$ contains at least three reduced words with distinct last letter and a common fixed point $wx$, a contradiction.
\end{proof}

If $S$ satisfies the conditions of Lemma \ref{general-affine-ping-pong-lemma}, we say that $a$ and $b$ form a \emph{locally commutative pair}.

\subsection{Quantitative Ping-Pong Lemmas}

\label{sec-ping-pong}

To apply Lemma \ref{ping-pong-lemma-generic}, we need to be able to exhibit the \emph{ping-pong table} (the subsets of $X$) and the \emph{ping-pong players} $a$ and $b$. If $u\in\R^2\setminus\{0\}$, and $\epsilon>0$, let
\[
\Nn_\theta(u,\epsilon)=\{x\in\R^2\setminus\{0\}\,:\, d([x],[u])\leq \epsilon\}.
\]
If $x\in \GL(2,\R)$ is semisimple with eigenvalues $|\lambda_1|\geq|\lambda_2|$, we let $\rho(x)=|\lambda_2/\lambda_1|$.

\begin{proposition}[Global Ping-Pong I, \cite{Breuillard2011a}]\label{uniform-ping-pong}
Let $S\subset\SL(2,\Z)$ be a finite symmetric set containing $1$. Assume that there exist $N_1,\,N_2,\,N_3\in\N$ such that the following holds:
\begin{enumerate}[label=(\roman*)]
\item $\Lambda(S^{N_1})>2\,\|S\|$;
\end{enumerate}
Let $\Vv=\{u_1,\,u_2\}\subset\R^2$ be a set of non-colinear eigenvectors of $a\in S^{N_1}$ such that $\Lambda(a)=\Lambda(S^{N_1})$.
\begin{enumerate}[label=(\roman*),resume]
\item $\exists h\in S^{N_2}$ s.t. $d([u],[v])\geq \|S\|^{-N_3}$ for all $u,\,v\in\Vv\cup h\Vv$, $[u]\neq[v]$.
\end{enumerate}
Then, $a^\ell$ and $b^\ell=h a^\ell h^{-1}$ (where $b=hah^{-1}$) play ping-pong on $\R^2\setminus\{0\}$ provided that $\ell> 2N_2+4N_3$.
\end{proposition}

\begin{proof}
Since $a\in\SL(2,\Z)$, the assumption $\Lambda(a)>2$ implies that $a$ is semisimple and torsion-free, with real eigenvalues, so $a$ can be diagonalized in $\GL(2,\R)$. Let $u_1,\,u_2\in \R^2$ be two unit non-colinear eigenvectors of $a$, and let $g\in\GL(2,\R)$ be a matrix whose columns are the vectors $u_1,\,u_2$, so that $ge_i=u_i$, $i\in\{1,2\}$, where $e_1,\,e_2$ are the canonical basis vectors of $\R^2$. Note that $\|g\|^2\leq 2$. Let $h'=g^{-1}hg$. For $i\in\{1,2\}$, let $v_i=h'e_i$. If $\Ww=\Vv\cup h\Vv$, then $\Ww'=g^{-1}\Ww=\{e_1,e_2,h'e_1,h'e_2\}$. By \eqref{basic-estimates-distances-matrix}, for any $u,\,v\in\Ww'$, with $[u]\neq[v]$,
\[
d([u],[v])\geq \|g\|^{-2}\,d([gu],[gv])\, d([u_1],[u_2])
\geq \|S\|^{-2N_3}/2.
\]
Let $a'=g^{-1}ag=\diag(\lambda,\lambda^{-1})$, where $|\lambda|>2\,\|S\|$. Thus, for any $\ell\in\N$, $\rho((a')^\ell)=|\lambda|^{-2\,\ell}<2^{-2\,\ell}\,\|S\|^{-2\,\ell}$. Also let $b=hah^{-1}$ and $b'=g^{-1}bg$. Since $a'$ is diagonal, we easily obtain the following inequality:
\begin{equation}\label{semisimple-control-diagonal}
d([a'v],[e_1])\, d([v],[e_2])\leq \rho(a'),\quad \forall v\in \R^2\setminus\{0\}.
\end{equation}
Let $\epsilon_1,\,\epsilon_2>0$. We now consider the sets $A^+=\Nn_\theta([e_1],\epsilon_1)$ and $A^-=\Nn_\theta([e_2],\epsilon_1)$, and let $B^{\pm}=h'A^{\pm}$. If we put $\epsilon_1=\rho((a')^\ell)^{1/2}=|\lambda|^{-\ell}$, this immediately implies that $(a')^\ell(X\setminus A^-)\subseteq A^+$, and by our choice of $B^\pm$, that $(b')^\ell(X\setminus B^-)\subseteq B^+$; if we now put $\epsilon_2=\epsilon_1\,\|h'\|^2$, and require $\epsilon_2<\|S\|^{-2N_3}/2$, this will imply that the sets $A^{\pm}$, $B^{\pm}$ are disjoint, so $(a')^\ell$ and $(b')^\ell$ will play ping-pong by Lemma \ref{ping-pong-lemma-generic}. Since $\|h'\|\leq 
2\,\|S\|^{N_2+N_3}$, we have $\epsilon_2=|\lambda|^{-\ell}\,\|h'\|^2 \leq 2^{-(\ell-2)}\,\|S\|^{2N_2+2N_3-\ell}$. The condition for ping-pong is then satisfied if $2N_2+2N_3-\ell\leq -2N_3$, that is, $\ell\geq 2N_2+4N_3$.
\end{proof}

One of the difficulties in proving Theorem \ref{main-theorem} is to make Lemma \ref{general-affine-ping-pong-lemma} quantitative so as to control the dynamics of affine transformations. For a non-unipotent $g\in \SA(2,\R)$, we denote by $\varphi(g)\in\R^2$ its unique fixed point.

\begin{proposition}[Global Ping-Pong II]\label{uniform-ping-pong-affine}
Let $S\subset\SA(2,\Z)$ be a finite symmetric set containing $1$. Assume that there exist $N_1,\,N_2,\,N_3\in\N^*$ such that the following hold: 
\begin{enumerate}[label=(\roman*)]
\item $\Lambda(\theta(S)^{N_1})>2\,\|\theta(S)\|$;
\end{enumerate}
Let $a\in S^{N_1}$ be such that $\Lambda(\theta(a))=\Lambda(\theta(S)^{N_1})$, and  let $\cV=\{u_1,\,u_2\}\subset\R^2$ be a set of non-colinear eigenvectors of $\theta(a)\in \theta(S)^{N_1}$.
\begin{enumerate}[label=(\roman*),resume]
\item there exists $h\in S^{N_2}$ such that for $b=hah^{-1}$, we have $\varphi(b)\neq \varphi(a)$, and
\[
d([u],[v])\geq \|\theta(S)\|^{-N_3},\quad\forall u,\,v\in \Vv\cup \theta(h)\Vv\cup \{\varphi(b)-\varphi(a)\},\,[u]\neq [v].
\]
\end{enumerate}
Then, $a^\ell$ and $b^\ell=ha^\ell h^{-1}$ generate a non-abelian free group $\mathrm{F}_2$ whose action on $\R^2$ is locally commutative, as soon as $\ell\geq 20(N_1+N_2+N_3)$.
\end{proposition}

Whereas the analysis was rather straightforward in the linear case, to prove Proposition \ref{uniform-ping-pong-affine}, we will need a number of preliminary estimates. Let $e_1,\,e_2$ be the canonical basis of $\R^2$ and let $a\in\SA(2,\R)$ be such that $\theta(a)$ is diagonal. Let $h\in\SA(2,\R)$, $b=hah^{-1}$, and $S=\{a,a^{-1},b,b^{-1}\}$. Let $v_i=\theta(h)e_i$, $i\in\{1,2\}$, and let $\Vv=\{e_1,e_2\}$, $\Ww=\Vv\cup \theta(h)\Vv$. Then, $\theta(h)\cV$ is a set of non-colinear eigenvectors for $\theta(b)$. For $x\in\R^2$ and $\epsilon>0$, let
\[
\Nn_\tau(x,\epsilon)=\{z\in\R^2\,:\, \|z-x\|\leq \epsilon\}.
\]
For $W\subset\R^2$, let $W^c=\R^2\setminus W$. Consider the following subsets of $\R^2$: let $\epsilon_i,\,\delta_i,\,R_i>0$, $i\in\{1,2\}$, and let

\smallskip
\noindent
\begin{minipage}{.45\textwidth}
\centering
\includegraphics[page=1,scale=0.45]{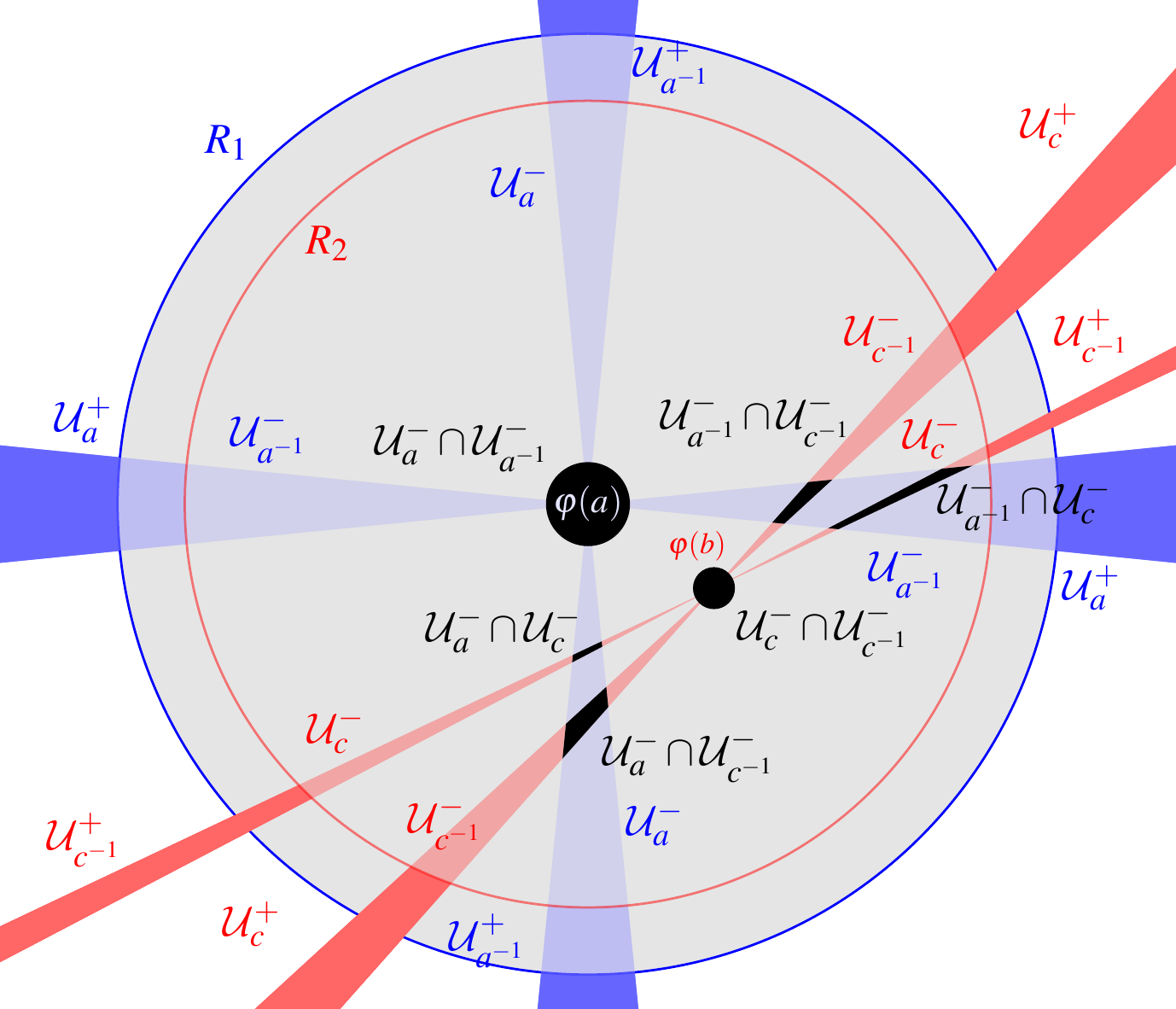}
\end{minipage}
\begin{minipage}{.45\textwidth}
\begin{align*}
A^+&=\Nn_\theta(e_1,\epsilon_1),\\
A^-&=\Nn_\theta(e_2,\epsilon_1),\\
C^+&=\Nn_\theta(v_1,\epsilon_2),\\
C^-&=\Nn_\theta(v_2,\epsilon_2),\\
\Uu_{a^{\pm 1}}^+&=[\varphi(a)+A^{\pm}]\cap \Nn_\tau(\varphi(a),R_1)^c,\\
\Uu_{a^{\pm 1}}^-&=[\varphi(a)+A^{\mp}]\cup \Nn_\tau(\varphi(a),\delta_1),\\
\Uu_{c^{\pm 1}}^+&=[\varphi(b)+C^{\pm}]\cap \Nn_\tau(\varphi(a),R_2)^c,\\
\Uu_{c^{\pm 1}}^-&=[\varphi(b)+C^{\mp}]\cup \Nn_\tau(\varphi(b),\delta_2).
\end{align*}
\end{minipage}
\smallskip

Let us say that the ping-pong table is \emph{proper} if the six intersections $\Uu_x^-\cap \Uu_y^-$, $x\neq y$ are disjoint and contained in the ball $\Nn_\tau(\varphi(a),R_2)$, with $R_1\geq R_2$.

\begin{lemma}\label{affine-proper}
The ping-pong table is proper if:
\begin{align}
&\|\varphi(b)-\varphi(a)\|>2\,\max\{\delta_1,\delta_2\};\label{eq:proper-affine-condition-1}\\
&\min\big\{d([u],[v])\,:\, u,\,v\in\Ww,\,[u]\neq [v]\big\}> 2\,\max\{\epsilon_1,\epsilon_2\};\label{eq:proper-affine-condition-2}\\
&\min\big\{d([u],[\varphi(b)-\varphi(a)])\,:\, u\in\Ww\big\}>
\max_{\substack{i,\,j\in\{1,2\}\\ i\neq j}}
\epsilon_i+\frac{\delta_j}{\|\varphi(b)-\varphi(a)\|};
\label{eq:proper-affine-condition-3}\\
&R_1\geq R_2>\frac{\|\varphi(b)-\varphi(a)\|}{d([u],[v])-(\epsilon_1+\epsilon_2)},\quad\forall u,\,v\in\cW,\,[u]\neq[v].
\label{eq:proper-affine-condition-4}
\end{align}
\end{lemma}

\begin{proof}
Condition \eqref{eq:proper-affine-condition-1} shows that $\Nn_\tau(\varphi(a),\delta_1)\cap \Nn_\tau(\varphi(b),\delta_2)=\emptyset$.

Note that for any vectors $u,\,v\in\R^2\setminus\{0\}$ and distinct $z_1,\,z_2\in\R^2$, we have for any $z$ lying in the intersection of $z_1+\Nn_\theta([u],\epsilon_1)$ and $z_2+\Nn_\theta([v],\epsilon_2)$,
\begin{equation}\label{useful-estimates-1}
\|z-z_1\|\leq \frac{\|z_2-z_1\|}{d([u],[v])-(\epsilon_1+\epsilon_2)}.
\end{equation}
Indeed, by the triangle inequality,
\[
d([u],[v])
\leq \epsilon_1+\epsilon_2+d([z-z_1],[z-z_2]),
\]
so \eqref{useful-estimates-1} follows from the fact that
\begin{equation}\label{eq:upper-bound-angular-distance}
d([z-z_1],[z-z_2])\leq
\frac{\|z_2-z_1\|}{\|z-z_1\|}.
\end{equation}
Setting $z_1=\varphi(a)$ and $z_2=\varphi(b)$, we see that condition \eqref{eq:proper-affine-condition-4} implies that all intersections between neighborhoods of the axes are contained in the ball $\Nn_\tau(\varphi(a),R_2)$.

Let $v=v_1$, and $z\in \Nn_\tau(\varphi(a),\delta_1)\cap [\varphi(b)+C^+]$. We have $\|z-\varphi(a)\|\leq \delta_1$ and $d([z-\varphi(b)],[v])\leq \epsilon_2$. Let $v_0=\varphi(b)-\varphi(a)$. By the triangle inequality,
\[
d([v],[v_0])\leq d([v],[z-\varphi(b)])+d([z-\varphi(b)],[v_0]).
\]
Since
\begin{align*}
d([z-\varphi(b)],[v_0])&=\frac{\|(z-\varphi(b))\wedge (\varphi(b)-\varphi(a))\|}{\|z-\varphi(b)\|\cdot \|\varphi(b)-\varphi(a)\|}
\leq \frac{\|z-\varphi(a)\|}{\|\varphi(b)-\varphi(a)\|},
\end{align*}
we obtain
\[
d([v],[v_0])\leq \epsilon_2+\frac{\delta_1}{\|\varphi(b)-\varphi(a)\|}.
\]
Condition \eqref{eq:proper-affine-condition-3} shows that this is impossible, so $\Nn_\tau(\varphi(a),\delta_1)\cap [\varphi(b)+C^+]$ must be empty. Similar computations for the other sets show that all intersections of balls with opposite neighborhoods of the axes are empty, and hence that the intersections $\cU_x^-\cap \cU_y^-\cap \cU_z^-$ for distinct $x,\,y,\,z\in \{a,a^{-1},c,c^{-1}\}$ are empty.
\end{proof}

Now that we have some control on the ping-pong table, we show that we can control both the table and the dynamics simultaneously. Assume that the conditions of Lemma \ref{affine-proper} hold. We continue to assume that $\theta(a)$ is diagonal and set $\theta(a)=\diag(a_1,a_2)$, where $\theta(a)e_i=a_ie_i$ for $i\in\{1,2\}$. Let $B^{\pm}=\theta(h)A^{\pm}$, and
\[
\Uu_b^+=h\Uu_a^+,\quad
\Uu_b^-=h\Uu_a^-,\quad
\Uu_{b^{-1}}^+=h\Uu_{a^{-1}}^+,\quad
\Uu_{b^{-1}}^-=h\Uu_{b^{-1}}^-.
\]

\begin{lemma}[Ping-Pong Players]\label{a-dilates-a-lot}
The following hold:
\begin{enumerate}[label=(\roman*)]
\item We have $y(X\setminus \cU_y^-)\subseteq \cU_y^+$ for each $y\in\{a,a^{-1},b,b^{-1}\}$, provided that
\[
\text{(a)}~\rho(a)\leq \epsilon_1^2,\quad\text{and}\quad
\text{(b)}~R_1\leq |a_1|\epsilon_1\delta_2.
\]
\item $\Uu_b^\pm\subseteq \Uu_c^\pm$ and $\Uu_{b^{-1}}^\pm\subseteq \Uu_{c^{-1}}^\pm$ provided that
\begin{equation}\label{eq:comparison-table}
\text{(a)}~\|\theta(h)\|^2\leq \frac{\epsilon_2}{\epsilon_1},\quad
\text{(b)}~ R_2+\|\varphi(b)-\varphi(a)\|\leq \frac{R_1}{\|\theta(h)\|},\quad
\text{(c)}~ \|\theta(h)\|\leq \frac{\delta_2}{\delta_1}.
\end{equation}

\end{enumerate}
\end{lemma}

\begin{proof}
(i) follows from \eqref{semisimple-control-diagonal}
and the fact that since $\theta(a)$ is diagonal,
\begin{equation}\label{eq:semisimple-control-diagonal}
\frac{\|az-\varphi(a)\|}{\|z-\varphi(a)\|}\geq |a_1|\, d([z-\varphi(a)],[e_2]),\quad\forall z\neq\varphi(a).
\end{equation}

For (ii), it is clear that (a) $\Rightarrow$ $B^\pm\subseteq C^\pm$. Then, $
\Uu_b^+\subseteq \Uu_c^+$ and $\Uu_{b^{-1}}^+\subseteq \Uu_{c^{-1}}^+$ because (b) implies that if $\|z-\varphi(a)\|>R_1$,
\begin{align*}
\|hz-\varphi(a)\|
&\geq \frac{R_1}{\|\theta(h)\|}-\|\varphi(b)-\varphi(a)\|
\geq R_2.
\end{align*}
Finally, (c) implies that $\Uu_b^-\subseteq \Uu_c^-$ and $\Uu_{b^{-1}}^-\subseteq\Uu_{c^{-1}}^-$ because if $\|z-\varphi(a)\|\leq \delta_1$, $\|hz-\varphi(b)\|=\|\theta(h)(z-\varphi(a))\|\leq \|\theta(h)\|\,\delta_1$.
\end{proof}

In order to apply Lemma \ref{general-affine-ping-pong-lemma}, we need in addition \eqref{eq:dilation-length-one}. Assume that the conditions of Lemmas \ref{affine-proper} and \ref{a-dilates-a-lot} hold, and let us continue to assume that $\theta(a)$ is diagonal. Our choice of function $f:\R^2\to\R_+$ will be $f:z\mapsto \|z-\varphi(a)\|$.

\begin{lemma}[Norm dilation]\label{norm-dilation}
$S=\{a,a^{-1},b,b^{-1}\}$ satisfies \eqref{eq:dilation-length-one} if
\begin{align}
&\epsilon_1>|a_1|^{-1},\tag{i}\\
&8\sqrt{2}\,\|\theta(h)\|^7< |a_1|\,\epsilon_1\, \min\left\{1, \|\varphi(b)-\varphi(a)\|^{-1}\,\delta_1\right\}.
\tag{ii}
\end{align}
\end{lemma}

\begin{proof}
It is clear that (i) implies that \eqref{eq:dilation-length-one} holds for $\{a,a^{-1}\}$ by \eqref{eq:semisimple-control-diagonal}. For any $x\in\Aff(\R^2)$ with $\theta(x)$ semisimple with eigenvalues $x_1,\,x_2$ and $|x_1|>|x_2|$, but not necessarily diagonal, one can check that if $u_1,\,u_2\in\R^2$ are two non-colinear eigenvectors for $\theta(x)$, then, for any $z\neq \varphi(x)$,
\begin{equation}\label{useful-estimates-2}
\frac{\|xz-\varphi(x)\|}{\|z-\varphi(x)\|}
\geq \frac{d([u_1],[u_2])^2\,|x_1|\,d([z-\varphi(x)],[u_2])}{2\sqrt{2}},
\end{equation}
in lieu of \eqref{eq:semisimple-control-diagonal}. Indeed, let $\theta(g)$ be a matrix of unit non-colinear eigenvectors $\{u_1,u_2\}$ for $\theta(x)$, and let $\theta(x')=\diag(x_1,x_2)$ where $\theta(x')=\theta(g^{-1}xg)$.  
If $v\in\R^2\setminus\{0\}$, and $e_1$, $e_2$ denote the canonical basis vectors,
\[
\frac{\|\theta(x)v\|}{\|v\|}\geq
\frac{\|\theta(g)^{-1}\theta(x)v\|}{\|\theta(g)^{-1}\|\,\|v\|}
=\frac{\|\theta(x')\theta(g)^{-1}v\|}{\|\theta(g)^{-1}\|\,\|v\|}
\geq \frac{|x_1|\, d([\theta(g)^{-1}v],[e_2])}{\|\theta(g)^{-1}\|}.
\]
Then, the result follows from \eqref{basic-estimates-distances-matrix} and the fact that $xz-\varphi(x)=\theta(x)(z-\varphi(x))$.

Let $z\in X\setminus \Uu_b^-$. We have
\[
\frac{\|bz-\varphi(a)\|}{\|z-\varphi(a)\|}
=
\frac{\|bz-\varphi(a)\|}{\|bz-\varphi(b)\|}
\cdot
\frac{\|bz-\varphi(b)\|}{\|z-\varphi(b)\|}
\cdot
\frac{\|z-\varphi(b)\|}{\|z-\varphi(a)\|}.
\]
Since $h^{-1}z\in X\setminus\Uu_a^-$, by \eqref{useful-estimates-2} applied to $x=b$, and \eqref{basic-estimates-distances-matrix}, we obtain
\begin{align*}
\frac{\|bz-\varphi(b)\|}{\|z-\varphi(b)\|}&\geq \frac{d([v_1],[v_2])^2\, |a_1| \,d([z-\varphi(b)],[v_2])}{2\sqrt{2}}
\geq \frac{|a_1|\, \epsilon_1}{2\sqrt{2}\|\theta(h)\|^6},
\end{align*}
a lower bound for the middle factor. By the triangle inequality, either
\[
\|z-\varphi(a)\|\leq 2\|z-\varphi(b)\|,\quad\text{or}\quad
\|z-\varphi(a)\|\leq 2\|\varphi(b)-\varphi(a)\|.
\]
Since $\|z-\varphi(b)\|>\|\theta(h)\|^{-1}\delta_1$, we obtain a lower bound for the third factor:
\[
\frac{\|z-\varphi(b)\|}{\|z-\varphi(a)\|}\geq \frac{1}{2\|\theta(h)\|}\cdot \min\left\{1,\frac{\delta_1}{\|\varphi(b)-\varphi(a)\|}\right\}.
\]
Finally, by the triangle inequality, either
\[
\|bz-\varphi(b)\|\leq 2\|bz-\varphi(a)\|,\quad\text{or}\quad
\|bz-\varphi(b)\|\leq 2\|\varphi(b)-\varphi(a)\|.
\]
Under the assumptions of Lemma \ref{a-dilates-a-lot}, $bz\in \cU_b^+\subseteq \cU_c^+$, so $\|bz-\varphi(a)\|>R_2>\|\varphi(b)-\varphi(a)\|$, so the first factor is $\geq 1/2$. Putting the bounds together, we see that \eqref{eq:dilation-length-one} holds if $|a_1|\,\epsilon_1\,\delta_1>8\sqrt{2}\,\|\theta(h)\|^7\,\|\varphi(b)-\varphi(a)\|$. An analogous calculation yields the same condition for the norm dilation for $b^{-1}$, with $|a_1|$ replaced with $|a_2|^{-1}$.
\end{proof}

It is convenient to summarize all these conditions in a concise form. 

\begin{proposition}\label{nice-quantitative-affine-ping-pong}
Let $a\in\SA(2,\Z)$ be such that $\theta(a)=\diag(a_1,a_2)$. Let $h\in\SA(2,\Z)$ be in general position with respect to $a$ so that $h\varphi(a)\neq \varphi(a)$, and assume that there exists $0<\eta<1/1000$ such that
\begin{align}
&d([u],[v])>\eta,\quad\forall u,\,v\in \Vv\cup \theta(h)\Vv\cup\{h\varphi(a)-\varphi(a)\},\ [u]\neq [v],\\
&|a_1|>\eta^{-4} \|\theta(h)\|^{10}.
\end{align}
Then, $a$ and $b=hah^{-1}$ freely generate a free group $\rF_2$ whose action on $\R^2$ is locally commutative.
\end{proposition}

\begin{proof}
For $i\in\{1,2\}$, we let $\xi_i$ and $\gamma_i$ be defined by
\[
\xi_i=\frac{R_i}{\|\varphi(b)-\varphi(a)\|},\quad
\gamma_i=\frac{\delta_i}{\|\varphi(b)-\varphi(a)\|}.
\]
Then, we let $\epsilon_1,\,\gamma_1,\,\xi_1$ be defined in terms of $\epsilon_2,\,\gamma_2,\,\xi_2$ as follows:
\[
\epsilon_1:=\frac{\epsilon_2}{\|\theta(h)\|^2},\quad
\gamma_1:=\frac{\gamma_2}{\|\theta(h)\|},\quad
\xi_1:=(\xi_2+1)\|\theta(h)\|^2.
\]
Then, we let $\gamma_2=\epsilon_2$ and $\xi_2=\epsilon_2^{-1}$, and then, finally, $\eta=3\epsilon_2$. Then, under the hypotheses of the Proposition, all conditions of Lemmas \ref{affine-proper}, \ref{a-dilates-a-lot} and \ref{norm-dilation} hold, and by Lemma \ref{general-affine-ping-pong-lemma}, $a$ and $b=hah^{-1}$ generate $\rF_2$ such that $\rF_2\curvearrowright\R^2$ is locally commutative.
\end{proof}

\begin{proof}[Proof of Proposition \ref{uniform-ping-pong-affine}]
It is clear that $\theta(a)$ is semisimple. Let $g\in\Aff(\R^2)$ be such that $\theta(g)$ is a matrix of unit non-colinear eigenvectors for $\theta(a)$, and $\tau(g)=\varphi(a)$. Then, $\|\theta(g)\|^2\leq 2$. We have $a'=g^{-1}ag=\diag(\lambda,\lambda^{-1})\in\SL(2,\R)$, $|\lambda|>2\|\theta(S)\|$. Then, $(a')^\ell=\diag(\lambda^\ell,\lambda^{-\ell})$, and $\rho((a')^{\ell})=|\lambda|^{-2\ell}$. Let $h'=g^{-1}hg$. Let $\Ww=\Vv\cup\theta(h)\Vv\cup\{\varphi(b)-\varphi(a)\}$. We choose the affine ping-pong table with respect to the lines corresponding to $\Ww'=\theta(g)^{-1}\Ww$, and second center $z_0=g^{-1}v_0=\theta(g)^{-1} v_0$ (see the Figure). By \eqref{basic-estimates-distances-matrix}, for any $u,\,v\in\Ww'$ with $[u]\neq[v]$,
\begin{align*}
d([u],[v])&\geq \|\theta(g)\|^{-2}\,d([\theta(g)u],[\theta(g)v])\,d([u_1],[u_2])\\
&\geq \|\theta(S)\|^{-2N_3}/2>\|\theta(S)\|^{-(2N_3+N_1)}.
\end{align*}
Since $h'=g^{-1}hg$, we have $\|\theta(h')\|\leq 2\,\|\theta(S)\|^{N_2+N_3}$. By Proposition \ref{nice-quantitative-affine-ping-pong}, $(a')^\ell$ and $(b')^\ell=h'(a')^\ell h'{}^{-1}$ play ping-pong if there exists $\frac{1}{100}>\eta>0$ such that
\[
\|\theta(S)\|^{-(2N_3+N_1)}> \eta,\quad\text{and}\quad
2^\ell \|\theta(S)\|^{\ell}> \eta^{-4}\,2^{10}\|\theta(S)\|^{10(N_2+N_3)}
.
\]
Put $\eta=\|\theta(S)\|^{-m}$, for some $m\in\N$ to be determined. Since we may assume $\ell>10$, $(a')^\ell$ and $(b')^\ell=h'(a')^\ell h'{}^{-1}$ play ping-pong if $m> 2N_3+N_1$ and $\ell>4m+10(N_2+N_3)$, so we have a locally commutative ping-pong (Lemma \ref{general-affine-ping-pong-lemma}) as soon as $\ell> 20(N_1+N_2+N_3)$ by choosing $m=2N_3+N_1$.
\end{proof}

\section{Proof of the Uniform Affine Tits Alternative}\label{sec:proof-main-theorem-uniform-non-amenability}

In this section, we provide the final ingredient to prove Theorem \ref{main-theorem} and its corollaries -- \emph{arithmeticity} -- and prove Theorem \ref{main-theorem}.

\subsection{Separation Properties}

\begin{definition}[General position]
Let $a\in\GL(2,\R)$ be semisimple, with $\Vv=\{v_1,v_2\}\subset\R^2\setminus\{0\}$ a set of non-colinear eigenvectors for $a$. We say that $h\in\GL(2,\R)$ is in \emph{general (linear) position} with respect to $a$ if $[h\Vv]\cap [\Vv]=\emptyset$. If $a\in\Aff(\R^2)$ and $\theta(a)$ is semisimple, we say that $h\in\Aff(\R^2)$ is in \emph{general (affine) position} with respect to $a$ if $\theta(h)$ is in general position with respect to $\theta(a)$ and $h\varphi(a)\neq\varphi(a)$.
\end{definition}

Elements in general position do not share common eigenlines or fixed points. With this in hand, arithmeticity allows us to control the action of powers of these elements along with the positions of their eigenlines and fixed points.

\begin{proposition}[Arithmetic linear separation]\label{good-separation}
Let $N_1,\,N_2\in\N$. Then, there exists $N_3\leq 30N_1+2N_2$ such that for any finite symmetric subset $S\subset\SL(2,\Z)$ containing $1$ and generating a non-virtually solvable subgroup, if there exist a torsion-free semisimple element $a\in S^{N_1}$, and $h\in S^{N_2}$ in general position with respect to $a$, then $
d([u],[v])\geq \|S\|^{-N_3}$ for every $u,\,v\in \Vv\cup h\Vv$, $[u]\neq [v]$, where $\Vv=\{u_1,u_2\}$ is a set of non-colinear eigenvectors for $a$.
\end{proposition}

\begin{proof}
Since $a$ is torsion-free and semisimple, we have $|\tr(a)|\geq 3$, so $\Lambda(a)>2$. We may pick a set of non-colinear eigenvectors $\Vv=\{u_1,u_2\}\subset(\Z[\sqrt{\delta}])^2$, where $\delta=\tr(a)^2-4$ as follows. Since $2\lambda^{\pm 1}\in\Z[\sqrt{\delta}]$, we can choose $u=2(a_{12},\lambda-a_{11})$ and $v=2(a_{12},\lambda^{-1}-a_{11})$. Then, $\det(u,v)=4a_{12}(\lambda^{-1}-\lambda)=-4a_{12}\sqrt{\delta}$. If $x=\det(u,v)$, and we denote by $\sigma(x)$ its Galois conjugate, we have $x\sigma(x)\in\Z$ so $|x\sigma(x)|\geq 1$, and since $|\sigma(x)|\leq 4\|a\|\sqrt{\delta}\leq 8\|a\|^2\leq \|a\|^5$, we have $|x|\geq |\sigma(x)|^{-1}\geq \|a\|^{-5}$. By our choice of $u$ and $v$, we have $\max\{\|u\|,\|v\|\}\leq \|a\|^5$, and we deduce that $d([u],[v])\geq \|a\|^{-15}$. Since $h\in\SL(2,\Z)$, this immediately gives $d([hu],[hv])\geq \|h\|^{-2}\|a\|^{-15}$. Since $hu$ and $hv$ are eigenvectors for $b=hah^{-1}$, we can apply the same argument as above and find an upper bound for the Galois conjugate of $\det(hu,v)$ (resp. $\det(u,hv)$). We obtain $\min\{d([hu],[v]),d([u],[hv])\}\geq \|h\|^{-2}\|a\|^{-30}$. Since $a\in S^{N_1}$ and $h\in S^{N_2}$, the result follows with $N_3=30N_1+2N_2$.
\end{proof}

By increasing the power of $S$ if necessary, we now show that we can upgrade this linear separation to obtain ``affine'' separation.

\begin{proposition}[Arithmetic affine separation]\label{good-affine-separation}
Let $N_1,\,N_2\in\N$. Then, there exist integers $N_4\leq 4N_3+3N_1$ and $N_5\leq 16(N_1+N_2)(N_1+N_3)$ (where $N_3\in\N$ is the constant from Proposition \ref{good-separation}) such that the following holds. Let $S\subset\SA(2,\Z)$ be a finite symmetric set containing $1$ and generating a non virtually solvable subgroup which does not fix a point in $\Q^2$. Assume that there exist $a\in S^{N_1}$ such that $\theta(a)$ is torsion-free and semisimple, with non-colinear eigenvectors $\Vv=\{u_1,u_2\}$, and $h\in S^{N_2}$ in general (affine) position with respect to $a$. Assume in addition that $\Lambda(\theta(a))>2\,\|\theta(S)\|$. Let $b=hah^{-1}$, and
\begin{align*}
\Ww_1&=\{h\varphi(a)-\varphi(a)\}\cup \Vv\cup \theta(h)\Vv,\\
\Ww_2&=\{b^{N_4}\varphi(a)-\varphi(a)\}\cup \Vv\cup \theta(b^{N_4})\Vv,\\
\Ww_3&=\{a^{N_4}\varphi(b)-\varphi(b)\}\cup \theta(h)\Vv\cup \theta(a^{N_4})\theta(h)\Vv.
\end{align*}
Then, there exists $i\in\{1,2,3\}$ such that
\begin{equation}\label{eq:first-case}
d([u],[v])\geq \|\theta(S)\|^{-N_5},\quad 
\forall u,\,v\in\Ww_i,\, [u]\neq [v].
\end{equation}
\end{proposition}

\begin{proof}
By Proposition \ref{good-separation}, there exists $r_2=N_3\in\N$ (independent of $S$) such that
\[
\min \{d([u],[v])\,:\, u,\,v\in \Vv\cup\theta(h)\Vv,\, [u]\neq [v]\}\geq \|\theta(S)\|^{-r_2}.
\]
\begin{claim}\label{claim3}
If $m\geq 2r_2+N_1$, there exists $r_3(m)=r_2+(2N_1+4N_2)m$ such that
\begin{equation}\label{eq:second-alternative}
\min \{d([u],[v])\,:\, u,\,v\in \Vv\cup\theta(b^m)\Vv,\, [u]\neq [v]\}\geq \|\theta(S)\|^{-r_3(m)}.
\end{equation}
\end{claim}
\begin{proof}[Proof of Claim \ref{claim3}]\renewcommand{\qedsymbol}{$\blacksquare$}
For any $u\in\Vv$, $[\theta(b)^mu]=[u]$ if and only if $\theta(h)u$ is an eigenvector of $a$, which is not the case since $\theta(h)$ is in general position with respect to $\theta(a)$. Since $d([u_1],[u_2])$ is bounded below and $\|\theta(b)\|\leq \|\theta(h)\|^2\|\theta(a)\|$, we have
\begin{align*}
d([\theta(b)^mu_1],[\theta(b)^mu_2])&\geq \|\theta(b)\|^{-2m} d([u_1],[u_2])\\
&\geq \|\theta(S)\|^{-((2N_1+4N_2)m+r_2)},
\end{align*}
and \eqref{eq:second-alternative} follows. If $i\neq j$, then
\[
d([\theta(b)^mu_i],[u_j])\geq
|d([\theta(h)u_1],[u_j])-d([\theta(b)^mu_i],[\theta(h)u_1])|.
\]
It follows from a change of basis and \eqref{semisimple-control-diagonal} that for any $v\in\R^2\setminus\{0\}$,
\begin{equation}\label{semisimple-control}
d([\theta(b)^mv],[\theta(h)u_1])\, d([v],[\theta(h)u_2])\,d([\theta(h)u_1],[\theta(h)u_2])^2
\leq \|\theta(S)\|^{-2m},
\end{equation}
where we use the fact that $\Lambda(\theta(a))>2\|\theta(S)\|$. Then, using the fact that $\|\theta(S)\|^{N_1}>2$, we have
\begin{align*}
d([\theta(b)^mu_i],[u_j])&\geq
\|\theta(S)\|^{-r_2}-\|\theta(S)\|^{-2m+3r_2}\\
&\geq \|\theta(S)\|^{-r_2}/2&\text{(if $m\geq 2r_2+N_1$)},\\
&> \|\theta(S)\|^{-(r_2+N_1)}.
\end{align*}
Thus, \eqref{eq:second-alternative} holds with $r_3(m)= r_2+(2N_1+4N_2)m$. 
\end{proof}

If \eqref{eq:first-case} holds for $\cW_1$, there is nothing to prove, so assume it fails. Let $r_4=r_2+N_1$. By the triangle inequality, only one of the four inequalities
\begin{equation}
d([v_0],[\theta(h)^\epsilon u])< \|\theta(S)\|^{-r_4},\quad \epsilon\in\{0,1\},\ u\in\cV,
\end{equation}
may hold. Let us first consider the case where
\begin{equation}\label{eq:one-out-of-4}
d([v_0],[u])< \|\theta(S)\|^{-r_4},\quad \text{for some}~u\in\cV,
\end{equation}
but $d([v_0],[\theta(h)u])\geq \|\theta(S)\|^{-r_4}$ for $u\in\cV$.

\begin{claim}\label{claim4}
For any $m\geq 4r_2+3N_1$, 
\begin{equation}\label{eq:good-second-step}
d([b^{\pm m} \varphi(a)-\varphi(a)],[u])\geq \|\theta(S)\|^{-(r_2+N_1)},\quad\forall u\in\Vv.
\end{equation}
\end{claim}

\begin{proof}[Proof of Claim \ref{claim4}]\renewcommand{\qedsymbol}{$\blacksquare$}
If $v_0=\varphi(b)-\varphi(a)$, then $[b^{\pm m}\varphi(a)-\varphi(b)]=[\theta(b)^{\pm m}v_0]$. If $u\in\cV$,
\begin{align*}
d([b^m\varphi(a)-\varphi(a)],[u])&\geq
d([\theta(b)^mv_0],[u])-d([\theta(b)^mv_0],[b^m\varphi(a)-\varphi(a)]),
\end{align*}
so Claim \ref{claim4} will follow if we can find an upper bound for $d([\theta(b)^mv_0],[b^m\varphi(a)-\varphi(a)])$ and a lower bound for $d([\theta(b)^mv_0],[u])$. For the lower bound, note that by \eqref{semisimple-control}, 
\begin{align*}
d([\theta(b)^mv_0],[\theta(h)u_1])&
\leq \|\theta(S)\|^{-2m} \|\theta(S)\|^{2r_2} d([v_0],[\theta(h)u_2])^{-1}\\
&\leq\|\theta(S)\|^{-r_5(m)},
\end{align*}
where $r_5(m)=2m-(2r_2+r_4)$. Therefore, we obtain the desired lower bound:
\begin{align*}
d([\theta(b)^mv_0],[u])&\geq d([u],[\theta(h)u_1])-d([\theta(b)^mv_0],[\theta(h)u_1])\\
&\geq \|\theta(S)\|^{-r_2}-\|\theta(S)\|^{-r_5(m)}.
\end{align*}
For the upper bound, by \eqref{eq:upper-bound-angular-distance}, with $z=b^m\varphi(a)$, $z_1=\varphi(b)$ and $z_2=\varphi(a)$, by \eqref{useful-estimates-2},
\begin{align*}
d(b^m\varphi(a)-\varphi(a),\theta(b)^mv_0)&\leq\frac{\|\varphi(b)-\varphi(a)\|}{\|b^m\varphi(a)-\varphi(b)\|}\\
&\leq \frac{2\sqrt{2}}{d(\theta(h)u_1,\theta(h)u_2)^2|a_1|^m d(v_0,\theta(h)u_2)}\\
&\leq \|\theta(S)\|^{-r_5'(m)},
\end{align*}
with $r_5'(m)=m-(2r_2+r_4)$. Hence,
\begin{align*}
d([b^m\varphi(a)-\varphi(a)],[u])&\geq \|\theta(S)\|^{-r_2}-\|\theta(S)\|^{-r_5(m)}-\|\theta(S)\|^{-r_5'(m)}\\
&\geq \|\theta(S)\|^{-(r_2+N_1)},
\end{align*}
if $\min\{r_5(m),r_5'(m)\}\geq r_2+2N_1$. This holds, e.g., if $m\geq 4N_3+3N_1$. Upon replacing $b^m$ with $b^{-m}$ and interchanging $u_1$ and $u_2$, the same argument shows that $d([b^{-m}\varphi(a)-\varphi(a),[u])\geq \|\theta(S)\|^{-(r_2+N_1)}$ for every $u\in\cV$.
\end{proof}

Claims \ref{claim3} and \ref{claim4} show that  \eqref{eq:first-case} will hold for $\cW_2$, because from Claim \ref{claim4}, we can obtain the remaining bound: Since $b\in S^{N_1+2N_2}$, if $u\in\cV$, we have
\begin{align}
d([b^m\varphi(a)-\varphi(a)],[\theta(b)^mu])
&\geq \|\theta(S)\|^{-2m(N_1+2N_2)} d([b^{-m}\varphi(a)-\varphi(a)],[u])\nonumber\\
&\geq \|\theta(S)\|^{-(2m(N_1+2N_2)+r_2+N_1)}.\label{eq:final-claim-3-4}
\end{align}
Hence, for $m\geq 4N_3+3N_1$, we can take $N_4=m$ and $N_5=16(N_1+N_2)(N_1+N_3)$.

Finally, assume that
\begin{equation}
d([v_0],[\theta(h)u]< \|\theta(S)\|^{-r_4},\quad \text{for some}~u\in\cV,
\end{equation}
but that $d([v_0],[u])\geq \|\theta(S)\|^{-r_4}$ for all $u\in\cV$. Then, we may repeat arguments similar to Claim \ref{claim3} and \ref{claim4} and obtain the following claims.  Since the calculations are similar, we omit the details.

\begin{claim}\label{claim5}
For any $m\geq 2r_2+N_1$,
\begin{equation}\label{eq:third-alternative}
\min \big\{ d([u],[v])\,:\, u,\,v\in \theta(h)\cV\cup \theta(a^m)\theta(h)\cV,\, [u]\neq[v]\big\}\geq \|\theta(S)\|^{-(r_2+2mN_1)}.
\end{equation}
\end{claim}
%

\begin{claim}\label{claim6}
For any $m\geq 4r_2+3N_1$, 
\begin{equation}
d([a^{\pm m}\varphi(b)-\varphi(b)],[\theta(h)u])\geq \|\theta(S)\|^{-(r_2+N_1)},\quad\forall u\in\cV.
\end{equation}
\end{claim}

Similarly to \eqref{eq:final-claim-3-4}, we then obtain for any $u\in\cV$,
\[
d([a^m\varphi(b)-\varphi(b)],[\theta(a)^m\theta(h)u])\geq \|\theta(S)\|^{-(2mN_1+r_2+N_1)}
\]
Again, $N_4=4N_3+3N_1$ and $N_5=16(N_1+N_2)(N_1+N_3)$ work, and this shows that \eqref{eq:first-case} holds for $\cW_3$.
\end{proof}

\subsection{Proof of Theorem \ref{main-theorem}}\label{sec:proof-of-main-theorem}

Let $S\subset\SA(2,\Z)$ be a finite symmetric set containing $1$ and generating a non-virtually solvable group $\Gamma$ which does not have a global fixed point in $\Q^2$. First, let us show that $\Gamma$ is Zariski dense in $\SA(2,\R)$. Indeed, let $\H$ be the Zariski closure of $\Gamma$. If $\theta:\SA_2\to\SL_2$ denotes the canonical projection onto $\SL_2$, which is a morphism of algebraic groups, then the projection $\theta(\H)$ is an algebraic $\R$-subgroup of $\SL_2$. If $\theta(\H)$ were a proper subgroup of $\SL_2$, it would be virtually solvable, and hence, amenable, which implies that $\H$ is amenable, a contradiction. Therefore, $\theta(\H)=\SL_2$. Now, $\H\cap \R^2$ is a closed subgroup of $\G$, normalized by the action of $\SL_2$. Since the action of $\SL(2,\R)$ on $\R^2$ is irreducible, if $\Gamma$ is not Zariski dense, we must have $\H\cap \R^2=\{0\}=\ker(\theta|_{\H})$. So $\theta$ is an isomorphism of $\R$-algebraic groups between $\H$ and $\SL_2$. By \cite[Lemma C.3]{BreuillardGreenGuralnickTao2015}, the first cohomology group of $\SL_2$ acting on $\R^2$ is trivial, so $\H$ has a fixed point $x_0\in\R^2$, a contradiction since $\Gamma$ does not have a global fixed point. Hence, $\Gamma$ is Zariski dense in $\SA(2,\R)$.

Proposition \ref{arithmetic-spectral-radius-lemma} shows that, up to conjugating $S$ inside $\SA(2,\Z)$, there exists an absolute constant $N_1\in\N$ such that $\Lambda(\theta(S)^{N_1})>2\|\theta(S)\|$. Let $a_0\in S^{N_1}$ be such that $\Lambda(\theta(a_0))=\Lambda(\theta(S)^{N_1})$. Then, $\theta(a_0)\in\SL(2,\Z)$ must be torsion-free and semisimple. Let $\Vv=\{v_1,v_2\}\subset\R^2$ be a set of non-colinear eigenvectors for $\theta(a_0)$. Let $Y_0=\{h\in\SA(2,\R)\,:\, [\theta(h)\Vv]\cap [\Vv]\neq\emptyset~\text{or}~h\varphi(a_0)=\varphi(a_0)\}$; this is a proper subvariety of $\SA(2,\R)$, and since $\Gamma$ is Zariski dense, we have $\Gamma \not\subset Y_0$. By Lemma \ref{Eskin-Mozes-Oh}, there exists $N_2\in \N$ such that for any set $S$ containing $1$ and generating $\Gamma$, $S^{N_2}\not\subset Y_0$, i.e., there exists $h_0\in S^{N_2}$ in general (affine) position with respect to $a_0$. Let $b_0=h_0a_0h_0^{-1}$. Proposition  \ref{good-affine-separation} then shows that the conditions of the Ping-Pong Lemma (Proposition \ref{uniform-ping-pong-affine}) hold for one of the following:
\begin{enumerate}
\item $a=a_0$, and $b=b_0$; or
\item $a=a_0$, $h=b_0^{N_4}$, and $b=hah^{-1}$; or
\item $a=b_0$, $h=a_0^{N_4}$, and $b=hah^{-1}$.
\end{enumerate}
This shows that for $\ell$ large enough, $a^\ell$ and $b^\ell$ play ping-pong on $\R^2$ according to Proposition \ref{uniform-ping-pong-affine} and Lemma \ref{general-affine-ping-pong-lemma}. This concludes the proof of Theorem \ref{main-theorem}.\qed

\begin{remark}\label{remark-regular}
It is not hard to modify the proof of Theorem \ref{main-theorem} to obtain the following result. If $\Gamma<\SL(2,\Z)$ is not virtually solvable, then the action $\Gamma\curvearrowright\R^2\setminus\{0\}$ is uniformly non-amenable. This can be done by applying Proposition \ref{uniform-ping-pong} 
together with the argument of \S \ref{sec:proof-of-main-theorem} and the following lemma.
\end{remark}

\begin{lemma}
\label{spectral-gap-measure}
Let $G\curvearrowright (X,\frM)$ a measurable action. If $a$ and $b$ form a ping-pong pair, then $G\curvearrowright(X,\frM)$ is $(\{a,b\},1/2)$-non-amenable. If there exists $N\in\N$ such that for any finite symmetric set $S$ generating $G$, $S^N$ contains a ping-pong pair, then $G\curvearrowright (X,\frM)$ is uniformly non-amenable.
\end{lemma}

\begin{proof}
By the triangle inequality, if the action is $(S^N,\epsilon)$-non-amenable for some $S\subseteq G$ and $\epsilon>0$, then it is $(S,\epsilon/N)$-non-amenable, so the second statement follows from the first and the fact that we may assume $S=S^{-1}$ regarding $(S,\epsilon$)-non-amenability. Let $\mu$ be a finitely additive probability measure on $(X,\frM)$, and assume that $a$ and $b$ form a ping-pong pair. Assume by contradiction that $|\mu(aM)-\mu(M)|<\epsilon$ for every $M\in\frM$. Then, if $M=X\setminus A^-$, note that $aM\subseteq A^+\subseteq M$, so $M=(aM)\sqcup (M\setminus aM)$, and $\mu(aM)\leq \mu(A^+)\leq \mu(M)$. Thus, $\mu(A^+\sqcup A^-)
\geq  \mu(aM)+\mu(A^-)
=\mu(M)-\mu(M\setminus aM)+\mu(A^-)$. But $\mu(M\setminus aM)<\epsilon$ by the assumption, so $\mu(A^+\sqcup A^-)\geq 1-\epsilon$. The same analysis for $b$ and $M=X\setminus B^-$ shows that $\mu(B^+\sqcup B^-)>1-\epsilon$ as well. Therefore, $0=\mu((A^+\sqcup A^-)\cap (B^+\sqcup B^-))>1-2\,\epsilon$, which is a contradiction if $\epsilon\leq 1/2$. Thus, $G\curvearrowright(X,\frM)$ is $(\{a,b\},1/2)$-non-amenable.
\end{proof}

\begin{remark}
For locally commutative actions, we can do something similar. Propositions \ref{good-affine-separation}  and \ref{uniform-ping-pong-affine} together show that we can apply Lemma \ref{general-affine-ping-pong-lemma}, but the ping-pong players $a$ and $b$ satisfy the additional property that $\cU_x^+\subseteq \cU_y^-$ where $y=x^{-1}$ for every $y\in \{a,b,a^{-1},b^{-1}\}$. This combined with the conditions of Lemma \ref{general-affine-ping-pong-lemma} and \S \ref{sec:proof-of-main-theorem} show, using an argument similar to Lemma \ref{spectral-gap-measure}, that there exists $N\in\N$ such that for any finite symmetric set $S$ containing $1$ and generating a non-virtually solvable subgroup $\Gamma$ which does not have a global fixed point, there exist $a,\,b\in S^N$ such that the action $\la S\ra\curvearrowright\R^2$ is $(\{a,b,a^{-1},b^{-1}\},1/4)$-non-amenable, which implies that it is $(\{a,b\},1/4)$-non-amenable. This gives an alternative more direct proof of Corollary \ref{uniform-non-amenable-affine-action} which bypasses the use of Dekker's Theorem (Theorem \ref{Dekker-theorem}).
\end{remark}

\section{Uniform Kazhdan constants for $\SA(2,\Z)$ and related results}\label{sec:uniform-Kazhdan}

In this final section, we study Problem \ref{relative-uniform-Kazhdan} and describe partial progress using Theorem \ref{main-theorem}. We will prove Corollaries \ref{uniform-Kazhdan} and \ref{uniform-Kazhdan-bound-II}.

\subsection{Proof of Corollary \ref{uniform-Kazhdan}}

Let us reduce the proof of Corollary \ref{uniform-Kazhdan} to the affine case. Note that proper algebraic $\R$-subgroups of $\SL(2,\R)$ are virtually solvable. Now, let $S\subset \SA(2,\Z)$ be a finite set containing $1$ and generating a non-virtually solvable subgroup $G$. Let $\G$ be the Zariski closure of $G$ in $\SA(2,\R)$, and let $H\leq G$ be a subgroup which is not Zariski dense in $\G(\R)$, with Zariski closure $\H$. If $H$ is amenable, then by the Hulanicki-Reiter Theorem and continuity of induction \cite[Theorems F.3.5 \& G.3.2]{BekkaHarpeValette2008}, $\lambda_{G/H}$ is weakly contained (in the Fell topology) in $ \lambda_G$. Thus, we may assume that $H$ is not amenable and not Zariski dense. Then, as in the proof of Theorem \ref{main-theorem} (see \S \ref{sec:proof-of-main-theorem}), $H$ fixes a point in $\R^2$. Let $x_0\in\R^2$ be this fixed point. Consider the map $\varphi : G/H \to Gx_0$ defined by $gH\mapsto gx_0$. Note that $\varphi$ is a $G$-equivariant isomorphism, and since $x_0$ may be viewed as the origin of the plane for the action $G\curvearrowright(G/H)$, by Corollary \ref{uniform-non-amenable-affine-action}, the action $G\curvearrowright(G/H)$ is uniformly non-amenable. It remains to prove the following general fact.

\begin{proposition}\label{uniformly-non-amenable-Kazhdan}
Let $X$ be a countable discrete set. If the action $G\curvearrowright X$ is $(S,\epsilon)$-non-amenable, then $\kappa_G(S,\pi_X)>\epsilon/2$.
\end{proposition}

\begin{proof}
Write $\pi=\pi_X$, and let $\xi\in \ell^2(X)$ be a unit vector. For $A\subseteq X$, let $\mu(A)=\sum_{x\in A} |\xi(x)|^2$. Then, $\mu$ is a probability measure on $X$. By the Cauchy-Schwarz inequality, we have for any $g\in G$ and $B\subseteq X$,
\begin{align*}
|g_*\mu(B)-\mu(B)|
&\leq\sum_{z\in B} \big|\pi(g)|\xi|^2(z)-|\xi|^2(z)\big|
\leq\|\pi(g)|\xi|^2-|\xi|^2 \|_1\\
&\leq
\big\|\pi(g)|\xi|+|\xi| \big\|_2
\big\|\pi(g)|\xi|-|\xi| \big\|_2
\leq 2\,\big\|\pi(g)|\xi|-|\xi| \big\|_2.
\end{align*}
In turn, by the (reverse) triangle inequality
\begin{align*}
\big\|\pi(g)|\xi|-|\xi| \big\|_2^2
&\leq\sum_{z\in X} \big|\pi(g)\xi(z)-\xi(z)\big|^2
=\|\pi(g)\xi-\xi\|_2^2.
\end{align*}
Thus, for any $B\subseteq X$ and $g\in G$, we have $
|g_*\mu(B)-\mu(B)|\leq 2\,\|\pi(g)\xi-\xi\|_2$.
\end{proof}

\subsection{Uniform Kazhdan Constants for Representations Coming from the Ambient Lie Group}

Here, we describe what can already be derived from the results of Burger \cite{Burger1991} and Breuillard and Gelander \cite{BreuillardGelander2008}. For an arbitrary locally compact group $G$, we denote by $\hat{G}$ the set of equivalence classes of irreducible unitary representations of $G$, the \emph{unitary dual} of $G$. Let $G=\SA(2,\Z)$ and $H=\SL(2,\Z)$. We first derive uniform Kazhdan constants from the uniform non-amenability of $\SL(2,\Z)\curvearrowright \R^2\setminus\{0\}$ (see Remark \ref{remark-regular}); this was the starting point for us to generalize the uniform non-amenability to the affine action. 

\begin{proposition}\label{Lie-group-uniform-spectral-gap}
There exists $\epsilon>0$ such that $\kappa_G(\sigma|_G)>\epsilon$ for any unitary representation $(\sigma,\cH)$ of $H\ltimes\R^2$ such that $\cH^{\R^2}=\{0\}$.
\end{proposition}

\begin{proof}
Let $P_\sigma$ be the projection valued measure given by the Spectral Theorem \cite[Theorem D.3.1]{BekkaHarpeValette2008} associated to the representation $\sigma|_{\R^2}$ and let $\xi\in\cH$ be a unit vector. Let $\mu_\xi$ be the probability measure on $\hat{\R}^2$ defined by $\mu_\xi(B)=\la P_\sigma(B)\xi,\xi\ra$ for any Borel set $B$. For any $g\in G$ and any unit vector $\xi\in\cH$, we have $\|g_*\mu_\xi-\mu_\xi\|_{\TV}\leq 2\,\|\sigma(g)\xi-\xi\|$ \cite{Burger1991}. For any finite generating set $S$ of $G$, $\theta(S)$ generates $H$, so
\[
\sup_{g\in S} \|\sigma(g)\xi-\xi\|\geq \frac{1}{2}\,\sup_{h\in \theta(S)}\|h_*\mu_\xi-\mu_\xi\|_{\TV}.
\]
Since $\sigma$ is unitary, we may assume without loss of generality that $S=S^{-1}$. The result follows from Remark \ref{remark-regular} since $\mu_\xi$ is a probability measure on $\R^2\setminus\{0\}$.
\end{proof}

In the proposition below, we analyze what uniformity can be derived from the argument of \cite{Burger1991} alone.

\begin{proposition}
There exists $\epsilon>0$ such that for any unitary representation $(\pi,\cH)$ of $G$ without $\Z^2$-invariant vectors, $\inf_S \kappa_G(S;\pi)>\epsilon$, where the infimum is taken over all finite generating sets containing $1$ with the property that for any $g\in S$ and any $v\in\Z^2$, we have: $\theta(g)[0,1)^2\cap ([0,1)^2+v-\tau(g))\neq\emptyset
~\Rightarrow~
(\theta(g),v)\in S$.
\end{proposition}

\begin{proof}
Let $\sigma=\Ind_G^{H\ltimes\R^2}(\pi)$ and let $\xi\in\cH_\pi$ with $\|\xi\|=1$. There is a natural map $\cH_\pi\to\cH_\sigma$, $\xi\mapsto f_\xi$ with $\|f_\xi\|_{\cH_\sigma}=\|\xi\|_{\cH_\pi}$, such that for any $g\in G$,
\[
\|\sigma(g)f_\xi-f_\xi\|^2=\sum_{v\in \Z^2} \|\pi(\theta(g),v)\xi-\xi\|^2\,m_{\R^2}\big[(\theta(g)\Lambda)\cap (\Lambda+v-\tau(g))\big],
\]
where $m_{\R^2}$ denotes the Lebesgue measure on $\R^2$  \cite[Proof of Proposition 1]{Burger1991}. Then, $\sigma$ has no $\R^2$-invariant vectors. By Proposition \ref{Lie-group-uniform-spectral-gap}, $\kappa_G(\sigma|_G)>\epsilon$ for some $\epsilon>0$ not depending on $S$. The result then follows because $\|\pi(\theta(g),v)\xi-\xi\|\leq \sup_{g\in S} \|\pi(g)\xi-\xi\|$ for each $(\theta(g),v)$ such that $(\theta(g)\Lambda)\cap (\Lambda+v-\tau(g))\neq \emptyset$.
\end{proof}

However, note that for such generating sets $S$, the set $S^{-1}S$ must contain unipotent elements (namely, the pure translations), because if there exist distinct $v_1,\,v_2\in\Z^2$ such that $g_1=(\theta(g),v_1),\ g_2=(\theta(g),v_2)\in S$, then $g_1^{-1}\,g_2$ is a pure translation in $S^{-1}S$. One way to construct such a generating set is to pick a generating set of $H$ and add all corresponding translations.

\subsection{Uniform Kazhdan Constants for $\SA(2,\Z)$}

In this final subsection, we will consider the explicit description of the unitary dual of $\SA(2,\Z)$ and give uniform Kazhdan constants for several (new) classes of irreducible unitary representations. This will use in an essential way the bound provided by Lemma \ref{split-tensor-induced} below and Theorem \ref{main-theorem} (or Corollary \ref{uniform-non-amenable-affine-action}). First, let us describe the irreducible representations of $\SA(2,\Z)$.

\subsubsection{Irreducible Representations of $\SA(2,\Z)$}\label{subsection:Mackey}

It follows from Mackey theory (see \cite[Lemmas 6.10, 6.22, 6.23, Theorem 6.11]{Varadarajan1985} and \cite{Thomas2015} for a different exposition) that we have the following description of the unitary dual of $G=\SA(2,\Z)$. For every $\pi\in\hat{G}$, there exist an $\SL(2,\Z)$-quasi-invariant ergodic measure $\mu$ on $\hat{\Z}^2$, which we identify with $\T^2=(\R/\Z)^2$, a separable Hilbert space $\cK$, and an irreducible unitary cocycle $\sigma\in \rH^1(\SL(2,\Z)\curvearrowright(\T^2,\mu),\cU(\cK))$, such that $\pi$ is equivalent to a representation $\pi_{\sigma,\mu}$ acting on $\frL_\mu^2(\T^2;\cK)$ defined by
\begin{equation}\label{eq:Mackey}
[\pi_{\sigma,\mu}(g)f](\chi)=\sqrt{c_\mu(h^{-1},\chi)}\chi(n)\sigma(h,h^{-1}\chi)f(h^{-1}\chi),
\end{equation}
where $\chi\in\T^2$, $n\in\Z^2$, $h\in \SL(2,\Z)$, $g=(h,n)$, $f\in \frL^2(\T^2;\cK)$, and $c_\mu$ is the Radon-Nikodym cocycle. Moreover, every such $\pi_{\sigma,\mu}$ is an irreducible unitary representation of $G$, and if $\pi_{\sigma,\mu}\simeq \pi_{\sigma',\mu'}$, then $\mu$ and $\mu'$ have the same measure class, and $\sigma$ and $\sigma'$ are cohomologous (in particular, if $\sigma$, $\sigma'$ are constant unitary cocycles, i.e., unitary representations, they are equivalent).

In general, these are intractable \cite{KatznelsonWeiss1972}, but of particular interest is the case where: (i) $\sigma$ is a unitary representation of $\SL(2,\Z)$ on $\cK$, i.e., a constant cocycle, and (ii) $\mu$ is invariant. For $G=\SA(2,\Z)$, this means that $\mu$ is either supported on an $\SL(2,\Z)$-orbit -- corresponding systems of imprimitivity are known as \emph{transitive} \cite{Varadarajan1985} -- or Lebesgue measure on $\T^2$ \cite{DaniKeane1979,Burger1991}. Fourier duality enables us to describe these irreducible representations more classically as induced representations; see \S \ref{subsection:natural-representation} and \S \ref{subsection:uniform-Kazhdan}.

\subsubsection{The Natural Representation of $\SA(2,\Z)$ on $\ell^2(\Z^2)$}\label{subsection:natural-representation}

The natural representation of $G=\SA(2,\Z)$ on $\ell^2(\Z^2)$ is equivalent to $\Ind_H^G(1_H)$, where $H=\SL(2,\Z)$. By Fourier duality, it is also equivalent to the representation $\pi_{1_H,\mathrm{Leb}}$ given by \eqref{eq:Mackey}. More generally, by Fourier duality, $\Ind_H^G(\sigma)\simeq \pi_{\sigma,\mathrm{Leb}}$, for any irreducible representation $\sigma$ of $H$. Note that \emph{Mackey's irreducibility criterion} \cite[Theorem 1.1]{BekkaCurtis2003} shows that $\Ind_H^G(\sigma)$ is irreducible if $\sigma$ is finite-dimensional, but as it turns out, it is irreducible even if $\sigma$ is infinite-dimensional, and we can derive a uniform Kazhdan bound for these representations. This is precisely the statement of Corollary \ref{uniform-Kazhdan-bound-II}.

\subsubsection{Herz' Majoration Principle and Proof of Corollary \ref{uniform-Kazhdan-bound-II}}

Much of our subsequent analysis for Kazhdan constants relies on operator norms of convolution operators. Let $G$ be a locally compact group and $(\pi,\cH_\pi)$ a unitary representation. If $\mu$ is a probability measure on $G$, we define the bounded linear operator $\pi(\mu)$ acting on $\cH_\pi$ by
\[
\pi(\mu)\xi=\int_G \pi(g)\xi\,\ud\mu(g),\quad\forall \xi\in\cH_\pi.
\]
If $(\sigma,\cH_\sigma)$ is another unitary representation of $G$ and $\pi$ is weakly contained in $\sigma$ (denoted $\pi\prec\sigma$), then $\|\pi(\mu)\|\leq \|\sigma(\mu)\|$ for any probability measure $\mu$ on $G$. The following proposition summarizes the relationship between $\|\pi(\mu)\|$ and $\kappa_G(S,\pi)$, where $\mu=\mu_S$ is the uniform probability measure on $S$.

\begin{proposition}\label{convolution-Kazhdan}
Let $G$ be a countable group, and let $S\subset G$ be a finite subset. If $\mu=\mu_S$ is the uniform probability measure on $S$, then, $\kappa_G(S,\pi)\geq 1-\|\pi(\mu_S)\|$. If in addition, $1\in S$ and $S=S^{-1}$, then $\|\pi(\mu_S)\|\leq 1-\kappa_G(S,\pi)^2/(16\card(S))$.
\end{proposition}

\begin{proof}
The first bound is clear; for the second, see \cite[Proposition 6.2.1]{BekkaHarpeValette2008}.
\end{proof}

In particular, if $\card(S)\leq 6$, this shows that
\begin{equation}\label{eq:weak-containment-substitute-Kazhdan-constants}
\|\pi(\mu_S)\|\leq \|\sigma(\mu_S)\|
\quad\Rightarrow\quad
\kappa_G(S,\pi)\geq \kappa_G(S,\sigma)^2/100.
\end{equation}

\begin{lemma}\label{split-tensor-induced}
Let $G$ be a discrete group and $H\leq G$. Let $\pi,\,\sigma$ be two representations of $H$. Then, $\Ind_H^G(\pi\otimes\sigma)$ is weakly contained in $\Ind_H^G(\pi)\otimes\Ind_H^G(\sigma)$. In particular,
\begin{equation}\label{eq:Herz}
\big\|(\Ind_H^G(\sigma))(\mu)\big\|\leq \|\lambda_{G/H}(\mu)\|,
\end{equation}
for every probability measure $\mu$ on $G$.
\end{lemma}

\begin{proof}
Since $G$ is discrete, $\pi<\big(\Ind_H^G(\pi)\big)\big|_H$, so as $H$-representations, $\pi\otimes\sigma<\big(\Ind_H^G(\pi)\big)\big|_H\otimes\sigma$. By continuity of induction \cite[Theorem F.3.5]{BekkaHarpeValette2008} and Mackey's tensor product theorem \cite[Theorem 2.58]{KaniuthTaylor2013}, $\Ind_H^G(\pi\otimes\sigma)\prec \Ind_H^G(\pi)\otimes\Ind_H^G(\sigma)$. Moreover, $\sigma\simeq \sigma\otimes 1_H$, so $\Ind_H^G(\sigma)\prec \lambda_{G/H}\otimes \Ind_H^G(\sigma)$, and the result follows.\qedhere
\end{proof}

Inequality \eqref{eq:Herz} is known as \emph{Herz' majoration principle} \cite{Herz1970}.

\begin{proof}[Proof of Corollary \ref{uniform-Kazhdan-bound-II}]
Let $G=\SA(2,\Z)$ and $H=\SL(2,\Z)$, and assume that $1\in S$ and $S=S^{-1}$. Let $(\sigma,\cH_\sigma)$ be a unitary representation of $H$, and let $\pi=\Ind_H^G(\sigma)$. By Corollary \ref{uniform-non-amenable-affine-action} and Proposition \ref{uniformly-non-amenable-Kazhdan}, there exist $N\in\N$ and $\epsilon_1>0$ such that $S^N$ contains two elements $a$ and $b$ such that $\kappa_G(\{a,b\},\lambda_{G/H})>\epsilon_1$, and since $a,\,b\in S^N$, by the triangle inequality, we also have $\kappa_G(S,\pi)\geq \kappa_G(\{a,b\},\pi)/N$. If $Q=\{1,a,b,a^{-1},b^{-1}\}$, so that $\card(Q)\leq 6$, and $\mu_Q$ is the uniform probability measure on $Q$, by \eqref{eq:Herz}, we have $\|\pi(\mu_Q)\|\leq \|\lambda_{G/H}(\mu_Q)\|$, and by \eqref{eq:weak-containment-substitute-Kazhdan-constants}, we have $\kappa_G(Q,\pi)=\kappa_G(\{a,b\},\pi)\geq \kappa_G(\{a,b\},\lambda_{G/H})^2/100$. Since all representations are unitary, this implies that for every finite generating set $S$,
\[
\kappa_G(S,\pi)\geq\frac{\kappa_G(\{a,b\},\pi)}{N}\geq \frac{\kappa_G(\{a,b\},\lambda_{G/H})^2}{100N}
>\frac{\epsilon_1^2}{100N}.
\]
So there exists $\epsilon_2>0$ such that for every finite generating set $S$ and every unitary representation $\sigma$ of $H$, we have $\kappa_G(S,\Ind_H^G(\sigma))>\epsilon_2$.
\end{proof}

\subsubsection{Uniform Kazhdan Constants for $\SA(2,\Z)$}\label{subsection:uniform-Kazhdan}

Let $\Gamma$ be a non-virtually solvable subgroup of $\SL(2,\Z)$ and let $G=\Gamma\ltimes\Z^2$. In this paragraph, we study the following special class of irreducible representations of $G$. We will restrict to those representations $\pi_{\sigma,\mu}$ (as described in \S \ref{subsection:Mackey} when $\Gamma=\SL(2,\Z)$) for which the quasi-invariant measure $\mu$ is supported on a single $\Gamma$-orbit $\Gamma\chi_0$ in $\T^2$. We may identify $\Gamma\chi_0$ with with the coset space $\Gamma/\Gamma_0$, where $\Gamma_0$ is the stabilizer of the point $\chi_0$ in $\T^2$. It turns out (see \cite[Chapters V \& VI]{Varadarajan1985} and \cite[Chapter 6]{Folland2016}) that these irreducible representations can be described in the following more familiar way: $\chi_0$ extends to a character $\tilde{\chi}_0$ of $G_0=\Gamma_0\ltimes\Z^2$ by setting $\tilde{\chi}_0(g)=\chi_0(n)$  for every $g=(h,n)\in\Gamma_0\ltimes\Z^2$. On the other hand, every $\rho\in\hat{\Gamma}_0$ can be lifted to $\tilde{\rho}\in\hat{G}_0$ in the obvious way, and then  $\tilde{\chi}_0\otimes\tilde{\rho}\in\hat{G}_0$. For a subset $A\subseteq\T^2$, let
\begin{equation}\label{eq:subset-unitary-dual}
\cS_A=\big\{\Ind_{G_0}^G(\tilde{\chi}_0\otimes\tilde{\rho})\,:\, \chi_0\in A,\ \rho\in\hat{\Gamma}\big\}\subset \hat{G},
\end{equation}
and then, let $\cS_f=\cS_{(\Q^2/\Z^2)\setminus\{0\}}$ and $\cS_\infty=\cS_{\T^2\setminus (\Q^2/\Z^2)}$. Note that $\cS_{\{0\}}=\hat{\Gamma}$.

Each $\pi\in\cS_{\T^2}$ is equivalent to some $\pi_{\sigma,\mu}\in\hat{G}$, where $\mu$ is counting measure supported on the coset space $\Gamma/\Gamma_0$, and $\sigma:\Gamma\times(\Gamma/\Gamma_0)\to \cU(\cK)$ is the $\Gamma$-cocycle defined as follows: let $s:\Gamma/\Gamma_0\to\Gamma$ be a cross section for the projection $\Gamma\to\Gamma/\Gamma_0$. Then, let $\beta(\gamma,x)=s(\gamma x)^{-1}\gamma s(x)$ for every $\gamma\in\Gamma$ and $x\in \Gamma/\Gamma_0$, and then $\sigma=\rho\circ\beta\in \mathrm{H}^1(\Gamma\curvearrowright(\Gamma/\Gamma_0),\cU(\cH_\rho))$ is the corresponding cocycle.

Note that since Lebesgue measure on $\T^2$ and counting measure on a single $\Gamma$-orbit are inequivalent, $\pi_{\sigma,\Leb}\notin \cS_{\T^2}$ for any $\sigma\in\hat{\Gamma}$, so in particular, the natural representation of $G$ on $\ell^2(\Z^2)$ does not lie in $\cS_{\T^2}$.

By Lemma \ref{split-tensor-induced}, for any probability measure $\mu$,
\begin{equation}\label{eq:induced-bound-herz}
\|\pi(\mu)\|\leq 
\|\lambda_{G/G_0}(\mu)\|,\quad \forall \pi\in\cS_{\T^2}.
\end{equation}
A uniform lower bound for the Kazhdan constants of $\cS_\infty$ and for a subset of $\cS_f$ is available, as we now show.

\begin{proposition}
There exists $\epsilon>0$ such that $\kappa_G(\pi)>\epsilon$ for every $\pi\in \cS_\infty$.
\end{proposition}

\begin{proof}
Note that if $\chi_0\notin\Q^2/\Z^2$, then $\Gamma_0$ is amenable and in fact unipotent. Indeed, if $\gamma\in\Gamma_0$ is not unipotent, then the equation $\gamma x\equiv x \rmod \Z^2$ has a unique solution in $\R^2/\Z^2$ which actually belongs to $\Q^2/\Z^2$. Hence, if $\pi\in\cS_\infty$, then $G_0$ is amenable, and hence, $\lambda_{G/G_0}\prec \lambda_G$. By Corollary \ref{uniform-non-amenable-affine-action}, there exist $N\in\N$ and $\epsilon>0$ such that for any finite symmetric generating set $S$ containing $1$, there exist $a,\,b\in S^N$ such that $\kappa_G(\{a,b\},\lambda_{G/\Gamma})>\epsilon$, and $\kappa_G(S,\pi)\geq \kappa_G(\{a,b\},\pi)/N$. On the other hand, if $Q=\{1,a,b,a^{-1},b^{-1}\}$, then by \eqref{eq:Herz}, we have $\|\pi(\mu_Q)\|\leq \|\lambda_{G}(\mu_Q)\|\leq \|\lambda_{G/\Gamma}(\mu_Q)\|$, and by \eqref{eq:weak-containment-substitute-Kazhdan-constants},
\[
\kappa_G(S,\pi)\geq \frac{\kappa_G(\{a,b\},\pi)}{N}\geq \frac{\kappa_G(\{a,b\},\lambda_{G/\Gamma})^2}{100N}>\frac{\epsilon^2}{100N}.\qedhere
\]
So there exists $\epsilon_2>0$ such that for every finite generating set $S$ and every $\pi\in \cS_\infty$, we have $\kappa_G(S,\pi)>\epsilon_2$.
\end{proof}

Now, we turn to the finite dimensional representations. If $\cA\subseteq\N$, let us denote by $\cS_{f,\cA}\subseteq\cS_f$ the subset of representations of the form $\Ind_{G_0}^G(\tilde{\chi}_0\otimes\tilde{\rho})$ as in \eqref{eq:subset-unitary-dual} such that $(a/n,b/n)\in\Q^2/\Z^2$ is a representative of $\chi_0$ with $\gcd(a,b,n)=1$ and $n\in \cA$.

\begin{proposition}
There exist $\epsilon>0$ and a density one set of primes $\scrP_1\subset\scrP$ such that $\kappa_G(\pi)>\epsilon$, for every $\pi\in\cS_{f,\scrP_1}\subset\cS_f$.
\end{proposition}

\begin{proof}
Let $\pi\in \Ss_f$, and let $(a/n,b/n)\in\Q^2$ be a representative of $\chi_0$ with $\gcd(a,b,n)=1$. Then, the Fourier transform intertwines $\pi_0=\Ind_{G_0}^G(\tilde{\chi}_0)$ with a subrepresentation of $\pi_{n}^0$, where $\pi_n^0$ is the Koopman representation of $G$ on $\ell^2_0(X_n)$, with $X_n=(\Z/n\Z)^2$.

We claim that there exists $N\in\N$ such that for every finite symmetric generating set $S$ containing $1$, $S^N$ contains two elements generating a Zariski dense subgroup of $\SA(2,\R)$. Indeed, by Theorem \ref{main-theorem}, there exists $N\in\N$ such that for any such $S$, $S^N$ contains two elements $a$ and $b$ generating $\rF_2$ whose action on $\R^2$ is locally commutative. In particular, $\rF_2$ does not have a global fixed point, and by \cite[Lemma C.2]{BreuillardGreenGuralnickTao2015}, we deduce that $\rF_2$ is Zariski dense, and $\kappa_G(S,\pi)\geq \kappa_G(\{a,b\},\pi)/N$.

Now, $\pi_{n}^0$ is finite-dimensional, and actually factors through the quotient map $\varphi_n:\SA(2,\Z)\to \SA(2,\Z/n\Z)$, so for any probability measure $\mu$ on $G$, by Lemma \ref{split-tensor-induced}, we have $\|\pi(\mu)\|\leq\|\pi_0(\mu)\|\leq\|\pi_{n}^0(\mu)\|=\|\lambda_{n}^0((\varphi_n)_*\mu)\|$, where $\lambda_{n}^0$ is the natural representation of $\SA(2,\Z/n\Z)$ on $\ell^2_0(X_n)$. Combining \cite[Theorem 1.1]{BreuillardGamburd2010} and \cite[Theorem 1]{LindenstraussVarju2016}, there exist $\epsilon>0$ and a density one subset of primes $\scrP_1\subseteq\scrP$ such that for any $p\in\scrP_1$, we have $\|\lambda_p^0(\mu_{Q_p})\|\leq\|\lambda_{\SA(2,\F_p)}^0(\mu_{Q_p})\|\leq 1-\epsilon$, for any finite symmetric generating set $Q_p\subset\SA(2,\F_p)$ with $\card(Q_p)=4$, where $\mu_{Q_p}$ is the uniform probability measure on $Q_p$, and $\lambda_{\SA(2,\F_p)}^0$ is the restriction of the regular representation of $\SA(2,\F_p)$ to $\ell^2_0(\SA(2,\F_p))$. 

By \emph{Nori's Theorem} \cite[Theorem 5.1]{Nori1987}, given a subset $Q\subset G$ generating a Zariski-dense subgroup of $G$, for all but finitely many primes $p$, the reduction modulo $p$ of $Q$ gives a generating subset $Q_p$ of $\SA(2,\F_p)$. Indeed, let $\la Q\ra$ be a Zariski-dense subgroup of $\SA(2,\Z)$. Then, $\la Q\ra$ has no global fixed point, so there exist two affine transformations $(h_i,z_i)\in \la Q\ra$ with semisimple linear parts and distinct fixed points. By multiplying the fixed point equations by $\det(1-h_1)\det(1-h_2)\in\Z$, we obtain equations with integral coefficients. By reducing modulo $p$ for $p$ larger than all integral quantities involved, we obtain two affine transformations of $\SL(2,\F_p)$ with distinct fixed points, so $\varphi_p(\la Q\ra)$ has no global fixed point. Applying the Strong Approximation Theorem to $\theta(\la Q\ra)$ then shows that $\theta(\varphi_p(\la Q\ra))\cong \SL(2,\F_p)$. By \cite[Theorem E]{Nori1987}, the first cohomology group of $\SL(2,\F_p)$ acting on $\F_p^2$ is trivial for $p$ larger than a fixed constant. Arguing as in \S \ref{sec:proof-of-main-theorem}, it follows that if $\varphi_p(\la Q\ra)$ were a proper subgroup of $\SA(2,\F_p)$, it would be conjugate to $\SL(2,\F_p)$, and thus would have a global fixed point, a contradiction. Hence, $\varphi_p(\la Q\ra)=\SA(2,\F_p)$.

Thus, we may let $Q=\{a,a^{-1},b,b^{-1}\}\subset G$ and $Q_p=\varphi_p(Q)$ for every prime $p$. Then, for every $p\in\scrP_1$ and every $\pi\in\cS_{\scrP_1}$,
\[
\|\pi(\mu_{Q})\|\leq\|\pi_p^0(\mu_{Q})\|=\|\lambda_p^0(\mu_{Q_p})\|\leq\|\lambda_{\SA(2,\F_p)}^0(\mu_{Q_p})\|\leq 1-\epsilon.
\]
Then, by Proposition \ref{convolution-Kazhdan}, for every finite generating set $S$, we have
\[
\kappa_G(S,\pi) \geq \frac{\kappa_G(\{a,b\},\pi)}{N}=\frac{\kappa_G(Q,\pi)}{N}\geq \frac{1-\|\pi(\mu_{Q})\|}{N}\geq\frac{\epsilon}{N}.\qedhere
\]
\end{proof}

Let us conclude by pointing out that even if one could prove the existence of $\epsilon>0$ such that $\inf_{\pi\in\Ss_{\T^2}} \kappa_G(\pi)>\epsilon$, this would not be sufficient to answer Problem \ref{relative-uniform-Kazhdan} positively, because $\Ss_{\T^2}$ is only a proper subset of $\hat{G}$. For instance, as we explained in \S \ref{subsection:Mackey} and \S \ref{subsection:uniform-Kazhdan}, the representations $\Ind_H^G(\sigma)$, where $\sigma\in\hat{H}$, are equivalent to $\pi_{\sigma,\Leb}$ and are irreducible, but do not belong to $\Ss_{\T^2}$. Corollary \ref{uniform-Kazhdan-bound-II} provides a uniform Kazhdan bound for these representations.

\bibliographystyle{plain}
\bibliography{bibliography}

\end{document}